\documentclass[11pt]{amsart}
\usepackage{lipsum}
\usepackage{amsfonts}
\usepackage{amsmath}
\usepackage{scalerel}
\usepackage{subcaption}
\usepackage{amsthm}
\usepackage{epsf,latexsym}
\usepackage{amsfonts}
\usepackage{fontawesome}
\usepackage{tikz}
\usetikzlibrary{arrows,matrix,positioning}
\usetikzlibrary{cd}
\usepackage{etex}
\usepackage{amssymb}
\usepackage{epsf,latexsym}
\usetikzlibrary{arrows,matrix,positioning}
\usetikzlibrary{cd}
\usepackage{mathrsfs}
\usepackage{epstopdf}
\usepackage{verbatim}
\usepackage{ctable}
\usepackage{wrapfig}
\usepackage{booktabs}
\usepackage{xspace}
\usepackage{url}
\usepackage{mathrsfs}
\usepackage{graphicx}
\usepackage{epstopdf}
\usepackage{algorithmic}
\usepackage{verbatim}
\usepackage{ctable}
\usepackage[margin=3.0cm]{geometry}
\usepackage{hyperref}
\hypersetup{
	colorlinks=true,         % false: boxed links; true: colored links
	linkcolor=blue,          % color of internal links
	citecolor=blue,          % color of links to bibliography
	filecolor=blue,          % color of file links
	urlcolor=blue            % color of external links
}
\ifpdf
\DeclareGraphicsExtensions{.eps,.pdf,.png,.jpg}
\else
\DeclareGraphicsExtensions{.eps}
\fi
\usepackage{enumitem}
\setlist[enumerate]{leftmargin=.5in}
\setlist[itemize]{leftmargin=.5in}

\makeatletter
\@namedef{subjclassname@2020}{
	\textup{2020} Mathematics Subject Classification}
\makeatother

\usepackage{amsrefs}
\usepackage{arydshln}
\usepackage{epsf,latexsym,graphicx}
\usepackage[matrix,arrow,tips,curve]{xy}
\usepackage{etex}
\usepackage{xspace}
\usepackage{amssymb}
\usepackage{pifont}
\usepackage{verbatim}
\usepackage{epsfig}
\usepackage{pst-grad}
\usepackage{pst-plot}
\usepackage{cases}
\usepackage{lmodern}
\usepackage{multicol}
\usepackage[colorinlistoftodos,prependcaption,textsize=tiny]{todonotes}
\usepackage{xargs} 
\usepackage{bm}
\usepackage{hyperref}
\usepackage{enumitem}

\newcommand{\Z}{{\ensuremath{\mathbb{Z}}}}
\newcommand{\ZZ}{\Z}

\newcommand{\R}{{\ensuremath{\mathbb{R}}}}
\newcommand{\RR}{\R}

\newcommand{\DD}{{\ensuremath{\mathbb{D}}}}
\newcommand{\PP}{{\ensuremath{\mathbb{P}}}}

\newcommand{\calB}{{\ensuremath{\mathcal{B}}}}

\newcommand{\calD}{{\ensuremath{\mathcal{D}}}}

\newcommand{\calM}{{\ensuremath{\mathcal{M}}}}

\newcommand{\calV}{{\ensuremath{\mathcal{V}}}}
\newcommand{\calQ}{{\ensuremath{\mathcal{Q}}}}
\newcommand{\SSS}{{\ensuremath{\mathbb{S}}}}
\newcommand{\ST}{{\ensuremath{\mathrm{ST}}}}
\newcommand{\Sup}{{\ensuremath{\mathrm{Supp}}}}

\newcommand{\calP}{{\ensuremath{\mathcal{P}}}}

\newcommand{\bbx}{{\ensuremath{\boldsymbol{x}}}}

\newcommand{\bn}{{\ensuremath{\boldsymbol{n}}}}

\newcommand{\bd}{{\ensuremath{\boldsymbol{d}}}}
\newcommand{\bs}{{\ensuremath{\boldsymbol{s}}}}
\newcommand{\bt}{{\ensuremath{\boldsymbol{t}}}}
\newcommand{\calS}{{\ensuremath{\mathcal{S}}}}

\newcommand{\Toly}{\calV}

\newcommand{\Splsp}{\mathbb S}

\newcommand{\spanset}{\mathrm{span}\,}
\newcommand{\St}{\mathrm{star}\,}
\newcommand{\supp}{\mathrm{supp}\,}

\renewcommand\ge\geqslant
\renewcommand\geq\geqslant
\renewcommand\le\leqslant
\renewcommand\leq\leqslant
\newcommand{\eddi}{\Diamond}

\newtheorem{theorem}{Theorem}
\newtheorem{corollary}[theorem]{Corollary}

\newtheorem{lemma}[theorem]{Lemma}
\newtheorem{conj lemma}[theorem]{Conjectural Lemma}

\theoremstyle{definition}
\newtheorem{definition}[theorem]{Definition}
\newtheorem{remark}[theorem]{Remark}
\newtheorem{example}[theorem]{Example}
\newtheorem{proposition}[theorem]{Proposition}

\title{Completeness characterization of Type-I box splines}

\author[N. Villamizar]{Nelly Villamizar} 
\address{Nelly Villamizar\\
	Department of Mathematics\\
	Swansea University}
\email{n.y.villamizar@swansea.ac.uk}
\urladdr{\url{https://sites.google.com/site/nvillami}}
\author[A. Mantzaflaris]{Angelos Mantzaflaris}
\address{Angelos Mantzaflaris\\     
	Inria Sophia Antipolis--M\'editerran\'ee\\     
	Universit\'e C\^ote d'Azur}  
\email{angelos.mantzaflaris@inria.fr}
\urladdr{\url{https://www-sop.inria.fr/members/Angelos.Mantzaflaris/}}
\author[Bert J\"uttler]{Bert J\"uttler}
\address{Bert J\"uttler\\
	Institute of Applied Geometry\\     
	Johannes Kepler University Linz}  
\email{Bert.Juettler@jku.at}
\urladdr{\url{http://www.ag.jku.at}}

\begin{document}

\begin{abstract}
	We present a completeness characterization of box splines on three-directional triangulations, also called \emph{Type-I} box spline spaces, based on edge-contact smoothness properties.  
	For any given Type-I box spline, of specific maximum degree and order of global smoothness, 
	our results allow to identify the local linear subspace of polynomials spanned by the box spline translates. 
	We use the global super-smoothness properties of box splines as well as the additional super-smoothness conditions at edges to characterize the spline space spanned by the box spline translates. 
	Subsequently, we prove the completeness of this space space with respect to the local polynomial space induced by the box spline translates.
	The completeness property allows the construction of hierarchical spaces spanned by the translates of box splines for any polynomial degree on multilevel Type-I grids. 
	We provide a basis for these hierarchical box spline spaces under explicit geometric conditions of the domain.
\end{abstract}

\keywords{
Type-I box splines, contact edge-edge characterization, completeness of box spline spaces, hierarchical box spline spaces, kissing triangles, over-concave vertices.
}
\subjclass[2020]{
41A15, 65D07, 13D02
}

\maketitle

\section{Introduction}
\emph{Box splines} are locally supported piecewise polynomial functions defined on uniform grids. 
They were first introduced by de Boor and DeVore in \cite{deBoor} and are  considered a generalization of the univariate B-spline functions to the multivariate setting.
From a geometric point of view, box splines can be seen as density functions of the shadows of higher dimensional boxes and half-boxes \cite{Prautzsch2002Box}.
We remark that they can also be studied as a special case of the so-called simplex splines \cite{dahmen83}. 
Box splines possess a number of useful properties that make them well-suited for applications. 
For instance, it has been shown that box splines have small support (a few cells of the underlying grid), they are non-negative, form a partition of unity, and are refinable i.e., the box spline spaces on refined grids are nested
\cites{subdivbs83,deBoor2008}.  

Box splines can be defined from an arbitrary set of directional vectors in $\R^n$, but of particular interest for Geometric Design are box splines surfaces which are defined on uniform triangulations of the plane. 
In this article, we focus on \emph{type-I box splines}.
They are splines defined on three directional meshes that are commonly known as \emph{type-I triangulations} of $\R^2$.

From the rich literature on box splines, we mention a few monographs and survey articles which include \cites{Boor,deBoor2008,chui,Prautzsch2002Box,Lyche91}, and a few representative publications on three specific topics.  
Firstly, a
substantial number of results on the {\em approximation power} of box
splines is described in the literature,
e.g. \cites{Lyche2008,ron1993approximation}.  
Secondly, several publications discuss techniques for the {\em efficient manipulation} of box spline bases.  
A general stable evaluation algorithm is devised in \cite{stable}.  
In \cite{bbform09} the problem of efficient evaluation of box splines is addressed by making use of the local Bernstein representation of basis functions on each triangle. 
Also, {\em numerical integration} schemes, which are important for
applications, based on quasi-interpolation have been considered in
\cites{int2d,int3d}.
Recent applications of box splines include surface fitting \cite{kangchendeng}, and solving linear elasticity problems in isogeometric analysis \cite{Giannelli19}. 
In other areas of mathematics, the theory of box splines has been proved useful to compute the volume of polytopes, and to deal with the integration of continuous functions over polytopes \cite{xu2011}.

In this article, we are interested in the linear spaces of spline functions generated by the the translates of any fixed type-I box spline. 
These spline spaces share good approximation properties. 
For example, the set of translates of any type-I box spline form a partition of unity on $\R^2$, they are globally, and also locally, linearly independent.
These properties of type-I box splines were studied by Dahmen and Micchelli in \cite{dahmen85a} and Jia in \cite{jia84}. 
In particular, they investigated the linear independence of translates of a box spline in \cite{jia85} and \cite{dahmen85b}.

The linear independence property implies that the set of translates of any type-I box spline constitutes a basis for the spline space they span. 
In general this property is not satisfied for box spline functions associated to other uniform partitions, and that makes type-I box splines particularly relevant for applications. 
For instance, the set of translates of box splines defined from a set of four directional vectors, the so-called type-II box splines, are linearly dependent. 
For a concise treatment of type-II box splines, further references, and alternative proofs of box spline properties, see \cite{LaiSchu}*{Chapter 12} and \cite{chui}*{Chapter 2}.

A second interesting feature of type-I box splines is that, although for any fixed type-I box spline of degree $d$ and order of smoothness $r$, the box spline translates form a basis, in general these translates do not generate all possible piecewise polynomial functions of degree $\leq d$ and global smoothness $r$ over the three--directional mesh.  
More precisely, the box spline translates span a proper subspace of the space of $C^r$-continuous spline functions $\calS_d^r(G)$ of degree at most $d$ on a three--directional mesh $G$ (cf. Figure~\ref{fig:grid}), for any $d>1$.
If the domain is taken as an infinite grid $G$, or as an infinite collection of triangles in $G$, then both spaces are infinite dimensional. However, if the domain is restricted to a finite collection of triangles $\Omega$, as it is the usual setting in practice, then their finite dimension differ.
Explicit dimension formulas in terms of the combinatorics of the domain $\Omega$ are well known for both spaces, $\dim \calS_d^r(\Omega)$ can be computed using homological methods \cite{bn1}, or Bernstein-B\'ezier methods as in \cite{chui}*{Chapter 2}.

In this work, we provide a characterization for the space spanned by box splines translates based on supersmooth conditions across the edges of the underlying partition. 
We prove that for any fixed degree $d$ and order of global smoothness $r$, the space of splines satisfying these extra local smoothness conditions is precisely the space spanned by the translates of the corresponding type-I box spline. 
From this we deduce that the type-I box spline spaces are complete with respect to the local polynomial space induced by the box spline translates. 
The proof of this result uses the Fourier transform of box splines, as well as the algebraic properties of the Bernstein-B\'ezier representations of type-I box splines.
We generalise the classic definition of spline space in Definition \ref{def:newspline} to link the spline polynomial pieces to specific polynomial subspaces of $\R[x,y]$, and present the main result in the paper in Section \ref{sec:ccl}.

Furthermore, we apply the completeness characterization of box splines into the construction of hierarchical spline spaces based on local refinements of a type-I triangulation. 
Hierarchical splines constitute a well-established approach to
adaptive refinement in geometric modeling \cite{fb88} and numerical analysis \cites{musthasan11,Schillinger2012116,vgjs11}.
Hierarchical tensor-product spline spaces were introduced by Kraft in \cite{kraft1998} using a selection mechanism for B-splines. 
The method has been refined leading to spline basis with better approximation properties, such as the partition of unity property, strong stability and full approximation power \cites{THB,gjs13,msEffortless,zeng2015}. 
It has also been adapted to Powell-Sabin splines \cite{speleers2009}, Zwart-Powell elements and B-spline-type basis functions for cubic splines on regular grids \cite{Zore}. 
In an earlier article, we constructed a hierarchical basis for quartic $C^2$-continuous box splines \cite{vmj2015}. 
Quartic hierarchical box splines spaces have also been studied and used for surface fitting applications by Kang, Chen and Deng in \cite{kangchendeng} and \cite{kangdeng2015}.
Truncated hierarchical type-I box splines were considered in \cite{Giannelli17} and \cite{Giannelli19} in connection to isogeometric analysis applications. 
Other subdivision schemes has been explored in \cite{Gerot19}. A $C^1$-continuous scheme based on cubic half-box splines was presented in \cite{Sabin19}.

The results we present in this article generalize our previous work \cite{vmj2015} on quartic box splines. 
Our results apply to type-I box splines of any polynomial degree with no restriction on the symmetry of their support. 

The remainder of this paper is organized as follows. 
In Section~\ref{sec:prelim} we introduce the relevant notation for type-I triangulations, spline functions and the directional derivatives. 
Section~\ref{sec:BS} concerns the definition and properties of type-I box spline spaces. We define the space of translates and recall existing results on local and global smoothness of these functions. 
In Section~\ref{sec:ccl} we prove Lemma~\ref{lemma:ccl} which is the main result in the paper, and corresponds to the edge-contact characterization for type-I box splines.
In Section~\ref{sec:H} we construct the hierarchical type-I meshes and the corresponding  hierarchical box spline spaces. % for any given type-I box spline function.
This construction follows the approach presented in \cite{mokris} and \cite{vmj2015}.
We conclude the paper with some final remarks in Section~\ref{sec:conlusion}.  

\section{Preliminaries}\label{sec:prelim}

Throughout this article we assume that $G$ is the uniform type-I triangulation of the real plane $\RR^2$, see Figure \ref{fig:grid}.
This triangulation is obtained by drawing in the north-east diagonals in the bi-infinity grid with grid lines at the integers.
This triangulation of the plane is associated to three directional vectors, namely $\bm e_1=(1,0)$, $\bm e_2=(0,1)$ and $\bm e_3=(1,1)$, and therefore is also called a three directional mesh. 
Each line of $G$ is parallel to one of these vectors and go through the points of the integer grid $\ZZ^2$.

The collection $T$ of triangles in $G$ are considered as open sets in, and we denote by $E$ the set of all edges in $T$, and $V=\ZZ^2$ the set of vertices.
The set of edges is the disjoint union $E=E_1\sqcup E_2\sqcup E_3$,
where $E_i$ is the set of edges that are parallel to the vector $\bm e_i$. The edges in $E_i$ are called edges of \emph{type $i$}.
\begin{figure}[ht]
	\centering
	\includegraphics[width=3.5cm]{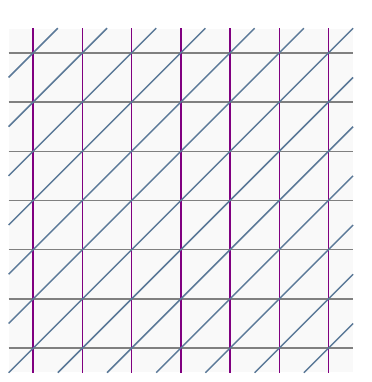}
	\caption{Uniform type-I triangulation (or three directional mesh) of $\RR^2$ associated to the directional  vectors $\bm e_1 = (0,1), \bm e_2 = (0,1)$ and $\bm e_3 = (1,1)$. We denote this grid as $G$.}\label{fig:grid}
\end{figure}
The combinatorial closure of a triangle $\triangle\in T$, denoted by $\hat
\triangle$, is the set consisting of the vertices and edges of
$\triangle$, and $\triangle$ itself. 
Analogously, $\hat \varepsilon$ of an edge $\varepsilon\in E$ is the set consisting of the edge itself and its two vertices.

A multicell domain $M$ is the triangulation in $\R^2$
induced by a finite set of triangles 
$\{\triangle_1,\, \dots,\, \triangle_m\}\subset T$, i.e.
\begin{equation*}
	M = %\bigcup_{\triangle\in M} \overline \triangle
	\bigcup_{i=1}^m \hat \triangle_i \ .
\end{equation*}
This means, that 
for every triangle $\triangle_i\in M$,
all the vertices and edges of $\triangle_i$ are considered as elements
of $M$.
The subspace of $\R^2$ defined by the (topological) closure $\overline\triangle_i$ of the triangles $\triangle_i$ defining the multicell domain $M$ will be denoted by $M^*$, namely
\begin{equation}\label{eq:M*}
	M^*=\bigcup_{\triangle\in M} \overline{\triangle}\ .
\end{equation}
Given a multicell domain $M$, the \emph{diamond of an edge $\varepsilon$} is
defined as the union of all (at most two) triangles of $M$ which have $\varepsilon$ as an edge, that is
\begin{equation*}
	\eddi(\varepsilon) = \bigcup_{\triangle\in M ,\, \varepsilon \in \hat{\triangle} } \hat \triangle \ . 
\end{equation*}
Similarly, the \emph{diamond of a vertex} $\nu$ is defined by
\begin{equation*}
	\eddi(\nu) = \bigcup_{\triangle\in M , \, \nu \in \hat{\triangle}} \hat \triangle \ , 
\end{equation*}
which is the union of the (at most six) triangles $\triangle$ in $M$
such that $\nu$ is a vertex of $\hat\triangle$.

Notice that the diamond $\diamond(\cdot)$, of an edge or a vertex, depends on the multicell domain $M$. From the context, it will be clear the particular domain we are considering in each case.

\medskip

We denote by $\R[x,y]$ the space of bivariate polynomials over the real numbers, and for $d\geq 0$,  $\calP_d\subseteq \R[x,y]$ is the set of all bivariate polynomials in $x$ and $y$ of total degree $\leq d$. 
In our presentation, the polynomial pieces that define the splines are taken from a finite vector subspace $\calV$ of $\R[x,y]$. This subspace $\calV$ is not necessarily the same as $\calP_d$ for any polynomial degree $d$, it may be a proper linear subspace. In this setting, we define the space of continuous splines $\PP(M, \calV)$ on a multicell domain $M$ as follows. 

\medskip

\begin{definition}\label{def:newspline}
	Given a multicell domain $M$, and a vector subspace $\Toly\subseteq\R[x,y]$,  we define  $\PP(M,\calV)$ as the {\em set of piecewise polynomials} functions on $M$ i.e.,
	\begin{equation*}
		\PP(M, \calV) = 
		\bigl\{
		f\in C^0(M^*) \colon   f|_\triangle\in\Toly|_\triangle \text{\, for each triangle } \triangle\in M
		\bigr\},
	\end{equation*}
	where $M^*$ is as defined in Equation \eqref{eq:M*}, and  $f|_\triangle$ denotes the restriction of the function $f$ to the triangle $\triangle$, and $\calV|_\triangle$ is the restriction to $\triangle $ of the polynomials in $\calV$ (seen as functions on $\R^2$).
\end{definition}

In particular, when $\calV=\calP_d$, the space $\PP(M,\calV)$ coincides with the usual space of $C^0$-continuous splines (or piecewise polynomial functions) on $M$ of degree at most $d$.

For any index $\bs=(s_1,s_2,s_3)\in\ZZ^3_{\ge0}$, that we also called \emph{regularity vector},  we consider the \emph{mixed
	directional derivative operator}
	\begin{align*}
		D_\bs\colon  \R[x,y]& \to\R[x,y]\\
		p\; &\mapsto   (\nabla \cdot \bm e_1)^{s_1}(\nabla \cdot \bm e_2)^{s_2}(\nabla \cdot \bm e_3)^{s_3}
		(p)\ .  
	\end{align*}
	For a given multicell domain $M$ and a polynomial vector space $\calV$, we extend the operator $D_\bs$ to elements $f\in\PP(M,\calV)$ by applying $D_\bs$ to
	the restrictions $f|_\triangle$, namely
	\begin{align*}
		(D_\bs f)|_\triangle = D_\bs (f|_\triangle)\ , \text{ for each triangle }\triangle\in M. 
	\end{align*}
	\begin{definition}\label{def:DI}
		For a given index set $I\subset\ZZ^3_{\ge0}$, and a vector space of functions $\calV\subseteq\R[x,y]$, we define the space of  
		functions $\DD_I(M,\calV)$ on a multicell domain $M$ by
		%We can now define the space of $D-$continuous functions.
		\begin{equation*}
			\DD_I(M,\calV) = \bigl\{
			f\in \PP(M,\calV) \colon  D_\bs f \in C^0({M^*}) \text{\; for all } \bs\in I
			\bigr\},
		\end{equation*}
		where $M^*$ is as defined in Equation \eqref{eq:M*}.
		%This space is obtained by requiring that the derivatives not only
		%exist formally but correspond to a smooth function on the closure of
		%the domain. 
	\end{definition}
	%\note{maybe remove this definition? look carefully if we need it!}
	\begin{remark}
		Using the notation in Definition \ref{def:DI} we have
		\begin{equation*}
			\DD_\emptyset(M,\calV) =\PP(M,\calV)\ .%= \DD_{\{(0,0,0)\}}(M,\calV)
		\end{equation*}
		If $\calS^r_d(M)$ denotes the space of globally $C^r$-continuous spline
		functions on $M$ of degree at most $d$, then $\calS^r_d(M)$ can be written as
		$\DD_I(M,\calP_d)$, where $I=\bigl\{\bm s\in\ZZ^3_{\ge0}\colon
		s_1+s_2+s_3\le r\bigr \}$. 
	\end{remark}
	
	In this paper, for a given multicell domain $M$, we shall consider piecewise polynomial functions, or splines, on $M$ with a specific order of smoothness associated to each of the three directions associated to the grid $G$.  Namely, for a regularity vector 
	$\bd=(d_1,d_2,d_3)\in\ZZ^3_{\ge0}$ we define the index sets:
	\begin{align*}
		I_{1}^\bd &= \bigl\{\bs\in \Z^3_{\geq 0}\colon s_2 + s_3\leq d_1 \bigr\},\\
		I_{2}^\bd &= \bigl\{\bs\in \Z^3_{\geq 0}\colon s_1 + s_3\leq d_2 \bigr\}, \\
		I_{3}^\bd &= \bigl\{\bs\in \Z^3_{\geq 0}\colon s_1 + s_2\leq d_3 \bigr\},
	\end{align*}
	and consider the spline space $\SSS^\bd(M,\calV)$ defined as follows. 
	
	\begin{definition}\label{def:edge-splines}
		For a multicell domain $M$ of the three directional grid $G$, a vector space $\calV\subseteq \R[x,y]$, and
		a vector $\bd\in \Z^3$, the \emph{spline space with edge smoothness $\bd$ on
			$M$} denoted $\SSS^\bd (M,\calV)$ is defined as 
		the set of piecewise polynomial functions on $M$ such that the derivatives of order
		$\bs\in D_i$ are continuous across the edges of type $i$ for
		$i=1,2,3$. More precisely,
		\begin{multline*}
			\SSS^\bd(M,\calV) = \bigl\{ f\in\PP(M,\Toly) \colon
			f|_{\eddi(\varepsilon)}\in C^{d_i}\left({\eddi(\varepsilon)^*}\right)
			\text{\, for every\, } \varepsilon\in E_i\cap M\\
			\text{and \;} i \in\{ 1,2, 3\}
			\bigr\}\ ,
		\end{multline*}
		where $\eddi(\varepsilon)^*=\bigcup_{\varepsilon\in\hat\triangle, \triangle\in M}\hat\triangle$ as defined in Equation \eqref{eq:M*}.
		
	\end{definition}

	Later in this paper (see Definition \ref{def:vertex-splines} below), we shall introduce a spline space but with smoothness conditions at the vertices of the domain $M$, the notation in Definition \ref{def:edge-splines} will be particularly convenient for that purpose. 
	%\end{remark}
	
	In the following example we illustrate Definition \ref{def:edge-splines} for a specific multicell domain in the grid $G$ and a regularity vector $\bm d\in \ZZ_{\ge 0}^3$. 
	\begin{example}\label{example-edges}
		Let $M$ be the multicell domain in Figure \ref{fig:example-edges}. 
		It is composed by four triangles denoted $\triangle_1,\dots,\triangle_4$, we take $\calV=\calP_2$ (the polynomials in $\R[x,y]$ of degree at most $2$), and the regularity vector $\bd=(0,1,0)$. Then
		\[
		I_1^\bd =\{(i,0,0)\colon i\in\Z_{\geq 0}\}, \; 
		I_2^\bd =\{(0,j,0),(1,j,0),(0,j,1)\colon j\in\Z_{\geq 0}\}\ ,\]
		\[ \text{and } \;
		I_3^\bd =\{(0,0,k)\colon k\in\Z_{\geq 0}\}\ .\]
		If we define $f\in\PP(M,\calP_2)$ by  $f|_{\hat \triangle_{i}}=f_i$, where 
		\begin{equation}\label{ex:f}
			\begin{aligned}[c]
				f_1 & =x^2; \\
				f_2 & = (x-y+1)^2; 
			\end{aligned} 
			\quad
			\begin{aligned}[c]
				f_3 & =(y-2x)(y-2); \\
				f_4 & =  2(x-1)(y-x) + 2x-y^2\ .
			\end{aligned}
		\end{equation}
		Then $f$  is an element in $\SSS^\bd(M,\calP_2)$. 
		In fact, if we put $g_{1,2}=(f_1,f_2)=f|_{\diamond(\varepsilon_1)}$, since $f_1-f_2=(y-1)(2x-y+1)$ then
		$D_\bs \, g_{1,2}$ is a continuous function on $\diamond(\varepsilon_1)$ for $\bs\in I_1^\bd$. Similarly, if we put $g_{i,i+1}=(f_i,f_{i+1})$ it is easy to check that $D_\bs \, g_{2,3}$ and $D_\bt \, g_{3,4}$ are continuous functions for every $\bs\in I_2^\bd$ and $\bt\in I_3^\bd$ on $\diamond(\varepsilon_2)$ and $\diamond(\varepsilon_3)$, respectively.  \hfill$\diamond$
		
		\begin{figure}
			\centering
			\includegraphics[width=4.5cm]{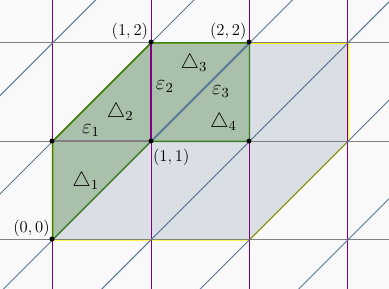}
			\caption{Multicell domain $M=\cup_{i=1}^4\hat\triangle_i$, with edges $\varepsilon_i=\overline\triangle_i\cap \overline\triangle_{i+1}$ for $i=1, 2, 3$.}\label{fig:example-edges}
		\end{figure}
	\end{example}
	%\cmb{Possible name: weakly-regular spline space of regularity $\bm d$ on $M$}
	In an analogous way as we defined a spline space associated to smoothness along the edges (Definition \ref{def:edge-splines}) of a multicell domain, we will introduce a space of splines with additional smoothness at the vertices of a given multivariate domain $M$. 
We prepare this definition by listing the possible vertex-vertex contact configurations    $\overline\triangle\cap\overline\triangle'=\nu$ between any pair of triangles
$\triangle,\triangle'\in T$. 
First we need the following definitions.
%These configurations lead to $4$ different smoothness types,which are shown in Figure \ref{fig:322}.
\begin{definition}\label{def:connected}
	Two triangles $\triangle$ and $\triangle'$ in the grid $G$, are said to be {\em edge-connected} if there is a %chain of triangles from $\triangle$ to $\triangle'$ in $G$, is a 
	collection of triangles
	$\triangle_0,\triangle_1, \dots, \triangle_m\in T$ such that $\triangle=\triangle_0$, 
	$\triangle'=\triangle_m$ and 
	$\overline\triangle_{i-1}\cap \overline\triangle_{i}\in E$ for every $i=1,\dots m$. 
	Such a collection of triangles $\triangle_0,\triangle_1, \dots, \triangle_m$ is called an \emph{edge-connected chain} between $\triangle$ and $\triangle'$.  
\end{definition}

\begin{definition}\label{def:smooth_type}
	If $\triangle,\triangle'\in T$ are triangles such that $\overline\triangle\cap\overline\triangle'\neq \emptyset$, we define the
	{\em smoothness type} $\ST(\triangle,\triangle')\subseteq\{1,2,3\}$ as the set of edge-types that are in the shortest edge-connected chain in $G$ between $\triangle$ and $\triangle'$. If $\triangle=\triangle'$ we define $\ST(\triangle,\triangle')=\emptyset$. 
\end{definition}
For any given pair of triangles  $\triangle,\triangle'\in T$ with a non-empty intersection, we can identify them  with a pair of triangles from $A$ to $F$ in Figure \ref{fig:321}, and their smoothness type $\ST(\triangle,\triangle')$ becomes one of the subsets listed in the table on the left of Figure \ref{fig:321}. 
\begin{figure}
	\hfill\begin{minipage}[c]{.45\linewidth}
		\scriptsize
		$\begin{array}{c|cccccc}
			\ST(\triangle,\triangle')  & A & B & C & D & E & F\\\hline
			A & \emptyset  &\{1\}          & \{1{,}2\}     & \{1{,}2{,}3\}&\{2{,}3\}     &\{3\}    \\
			B & \{1\}         & \emptyset  & \{2\}         & \{2{,}3\}    &\{1{,}2{,}3\} &\{1{,}3\}\\
			C & \{1{,}2\}     & \{2\}         & \emptyset  & \{3\}        &\{1{,}3\}     &\{1{,}2{,}3\} \\
			D & \{1{,}2{,}3\} & \{2{,}3\}     & \{3\}         & \emptyset &\{1\}         &\{1{,}2\}  \\
			E & \{2{,}3\}     & \{1{,}2{,}3\} & \{1{,}3\}     & \{1\}        & \emptyset &\{2\}     \\
			F & \{3\}         & \{1{,}3\}     & \{1{,}2{,}3\} & \{1{,}2\}    & \{2\}        &\emptyset \\
		\end{array}$    
	\end{minipage}\hspace{2.5cm}\begin{minipage}[r]{.20\linewidth}
		\includegraphics[width=2.5cm]{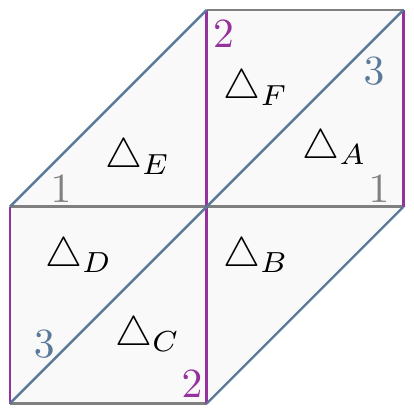}
	\end{minipage}\hfill\phantom{.}
	
	\caption{ Smoothness types of a pair of triangles
		$\triangle,\,\triangle'$ such that
		$\triangle\cap\triangle'\ne\emptyset$. We can identify
		$\triangle,\,\triangle'$ with two triangles in the picture on the
		right.  The type $\ST(\triangle,\triangle')$ is the
		corresponding index set shown in the table on the left side, which
		is constructed according to the shortest edge-connected chain
		between them (see Definition~\ref{def:connected}).}\label{fig:321}
\end{figure}

We now use Definitions \ref{def:connected} and \ref{def:smooth_type} to introduce the space of \emph{strongly regular splines} associated to a multicell domain $M$ in the three directional grid $G$.

\begin{definition}\label{def:vertex-splines}
	For a multicell domain $M\subset T$, a vector space $\calV\subseteq \R[x,y]$, and a regularity vector
	$\bd=(d_1,d_2,d_3)$, a spline $f\in \PP(M,\calV)$ is said to be \textit{strongly regular} if the for any pair of triangles $\triangle,\triangle'\in M$ such that $\triangle\cap\triangle'\neq \emptyset$, for every triple of indexes associated to the smoothness type $\ST(\triangle,\triangle')$ the derivatives of $f$ are $C^0$-smooth. 
	(Notice that the edge-connected chain are composed by triangles in the grid $G$ and are not necessarily in $M$.)
	The set of all strongly regular splines on $M$ will be denoted as $\hat{\SSS}^\bd(M, \calV)$, it is given by
	\begin{multline*}
		\hat{\SSS}^\bd(M, \calV) = \biggl\{f\in\PP(M,\calV)\,\colon\,
		f|_{U}\in\DD_{I}(U,\calV) 
		\text{ for }  \triangle,\triangle'\in M, \triangle\cap\triangle'\neq \emptyset\ ,\\
		U = \overline\triangle\cup\overline\triangle' \, , 
		{\mbox{ and $I=\bigcap_{i\in \ST(\triangle,\triangle')} I_{i}^\bd$}}
		\biggr\}.
	\end{multline*}
	The set $\hat{\SSS}^\bd(M, \calV)$ is the linear space of splines \emph{with edge and vertex   smoothness $\bd$ on the multicell domain $M$}.
\end{definition}

\medskip

\begin{example}\label{ex:propercontained}
	Let $M$ be the multicell domain in Figure \ref{fig:example-edges}, $\calV=\calP_2$ and $\bd=(0,1,0)$ as in Example~\ref{example-edges}.
	It is easy to check that the piecewise function $f$ defined in Equation \eqref{ex:f} is in $\hat{\SSS}^\bd(M, \calV)$. For instance, if we take the triangles $\triangle_1$ and $\triangle_4$, the smoothness type $\ST(\triangle_1,\triangle_4)=\{1,2,3\}$. Then   $I=\bigcap_{i=1}^3 I_i^{\bd}=\{(0,0,0)\}$, and in fact $f_1(1,1)=f_4(1,1)$. 
	Similarly, taking the triangles $\triangle_2$ and $\triangle_3$, we get $\ST(\triangle_2,\triangle_4)=\{2,3\}$ and $I=\bigcap_{i=2}^3 I_i^{\bd}=\{(0,0,k)\colon k=0,1\}$. The polynomials $f_2$ and $f_4$ and also their derivatives $\partial(f_i)/\partial (x-y)$, for $i=2 $ and $3$, have the same value at $(1,1)$.
	
	In contrast, if $g$ is the function on $M$ defined by $g|_{\triangle_i}=g_i$ with $g_1= 0$, $g_2= y-1$, $g_3= x^2-2x+y$ and $g_4= x^2 -y$, then $g$ is also in $\SSS^\bd(M,\calV)$, but it is not in $\hat{\SSS}^\bd(M, \calV)$. In fact, for the triangles $\triangle_2$ and $\triangle_4$, the derivatives $\partial g_2/\partial(x-y) = -1$ and $\partial g_4/\partial(x-y) = 2x + 1$. 
	Then $g|_{U}\not\in \DD_I(U,\calV)$ for $U=\overline\triangle_2\cup\overline\triangle_4$.\hfill $\Diamond$
\end{example}
\begin{remark}
	From Definitions \ref{def:edge-splines} and \ref{def:vertex-splines}, it is clear that both $\hat{\Splsp}^\bd(M,\Toly)$ and
	$\Splsp^\bd(M,\Toly)$ are contained in the space of splines that are globally
	$C^r$-continuous on $M$, where $r=\min\{d_1,d_2,d_3\}$.  
	Moreover, they are both contained in
	$\DD_{I}(M,\calV)$ for $I = I_{1}^\bd \cap I_{2}^\bd \cap
	I_{3}^\bd$ (see Definition \ref{def:DI}), and $\hat\SSS^\bd(M,\calV)\subseteq \SSS^\bd(M,\calV)$. By Example \ref{ex:propercontained} we also know that the set of strongly regular splines $\hat\SSS^\bd(M,\calV)$ may be properly contained in the spline space $\SSS^\bd(M,\calV)$.
\end{remark}
In the following we give a sufficient condition for the equality $\Splsp^\bd(M,\Toly)=\hat{\SSS}^\bd(M, \calV)$.  
For this we introduce the concept of over concave vertices and kissing triangles. 

\begin{definition}\label{def:oconcave}
	For a multicell domain $M$ in the three-directional grid  $G$ we say
	that a vertex $\nu \in V$ on the boundary of $M$ is \emph{over-concave}
	if $\St(\nu) \setminus M$ consists of a single triangle,
	where $\St(\nu)=\bigcup_{\substack{\triangle\in T,\\ \nu\in\hat\triangle }} \hat\triangle$\ . 
	
	Since we work exclusively in the three directional grid $G$ then $\St(\nu)$ consists of 6 triangles, 6 edges and $v$ itself, 
	see Figure~\ref{fig:overconcave}.
\end{definition}
\begin{figure}[ht]
	\centering
	\includegraphics[width=5cm]{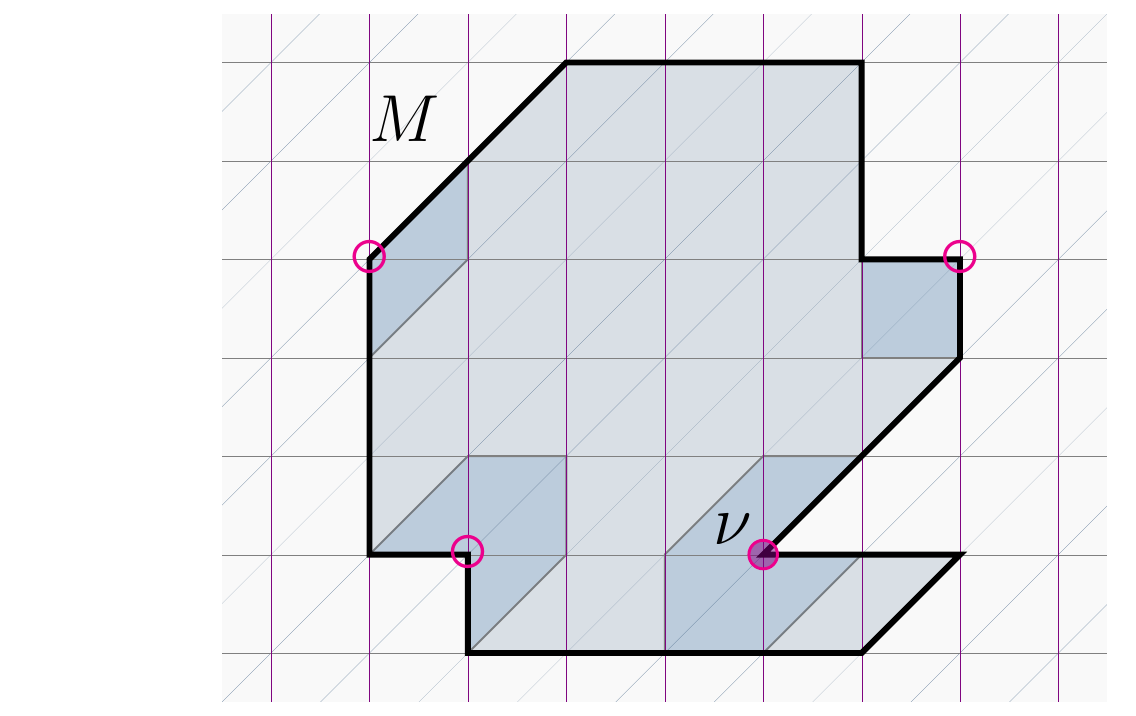}
	\caption{The vertex $\nu$ on the boundary of $M$ is an
		\emph{over-concave vertex} as in Definition \ref{def:oconcave}.}\label{fig:overconcave}
\end{figure}

\begin{definition}\label{def:kissing}
	Two triangles $\triangle$ and $\triangle'$ in $M$ such that $\triangle
	\cap \triangle'=\{v\}$ is a vertex $v\in V$ will be called \emph{kissing
		triangles}. 
\end{definition}
In Figure~\ref{fig:321}, for instance, the triangles $\{A,E\}$, $\{A,D\}$ $\{A,C\}$ are kissing triangles.

\begin{proposition}\label{prop:adm1}
	If $M$ be a multicell domain in the three directional grid $G$, such that it does not have kissing triangles nor over-concave boundary vertices, then  $\Splsp^\bd(M,\Toly)=\hat{\SSS}^\bd(M, \calV)$ i.e., all splines with edge smoothness $\bd$ on $M$ are strongly regular. 
\end{proposition}
\begin{proof}
	The statements follows directly from Definitions \ref{def:edge-splines}, \ref{def:vertex-splines} and \ref{def:oconcave}.
\end{proof}

\section{Box splines on type-I triangulations}\label{sec:BS}
In this section we define box splines on the uniform type-I triangulation $G$ defined in Section \ref{sec:prelim}, Figure \ref{fig:grid}. 
This triangulation has vertices at all lattice points $(i,j)\in\ZZ^2$.% we denote these points in the plane either by $\nu$ or $(x,y)$.
\begin{definition}\label{def:support}
	If $\beta$ is a real-valued function on $\R^2$, we denote by $\supp(\beta)$ the support $\beta$, and it is defined as the set of points $\bbx\in\R^2$ such that $\beta(\bbx)\neq 0$. 
\end{definition}
Recall from Definition \ref{def:oconcave}, that the \emph{star of a vertex} $\nu\in V$, denoted $\St(\nu)$, is defined by
\[\St(\nu)=\bigcup_{\triangle\in T, \, \nu\in\hat\triangle} \hat\triangle\ .
\]
It is the multicell domain composed by all the triangles $\triangle\in T$ which have $\nu$ as one of their vertices, together with the edges and vertices of these triangles. 
\begin{definition}\label{def:convolution}
	If $\bn=(n_1, n_2, n_3)\in \ZZ^3$ is a triple of integers such that $n_i \geq 1$, the type-I box spline $\calB_\bn$ associated to $\bn$ is defined recursively by
	\[
	\calB_\bn(\bbx) = \int_0^1 \calB_{\bn - \bm e_i}(\bbx-t\bm e_i)dt \ ,
	\]
	for $\bbx\in\R^2$ and $i\in\{1,2,3\}$ such that $\bn -\bm e_i\geq
	\boldsymbol{1}= (1,1,1)$; the function $\calB_{\boldsymbol 1}$ is the
	classical {\em Courant hat function} with support on the star of the
	vertex $(1,1)$ given in Figure \ref{fig:courant}. More precisely, $\calB_{\boldsymbol 1}$ is the
	piecewise linear function on $\R^2$ satisfying
	$\calB_{\boldsymbol 1}(1,1)=1$ and $\calB_{\boldsymbol 1}(i,j)=0$ for
	$(i,j)=(0,0), (0,1), (1,2), (2,1)$ and $(2,2)$.
\end{definition}
\begin{figure}[ht]
	\begin{center}
		\centering
		\includegraphics[width=3cm]{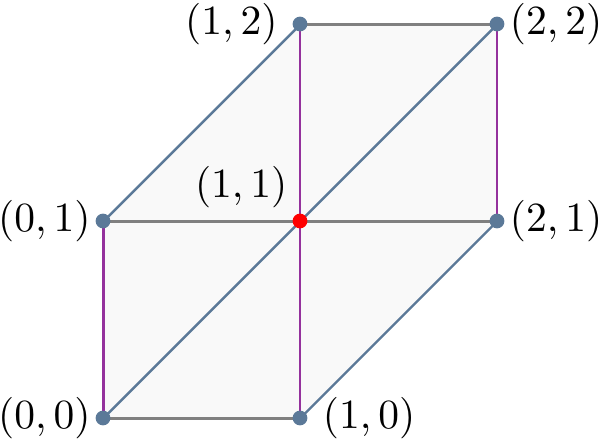}\quad\includegraphics[width=3.1cm]{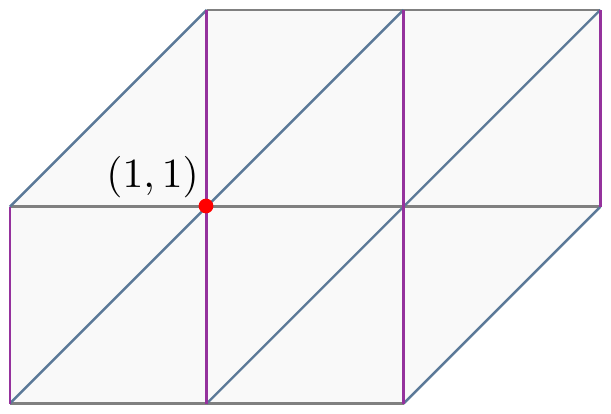}\quad\includegraphics[width=3.1cm]{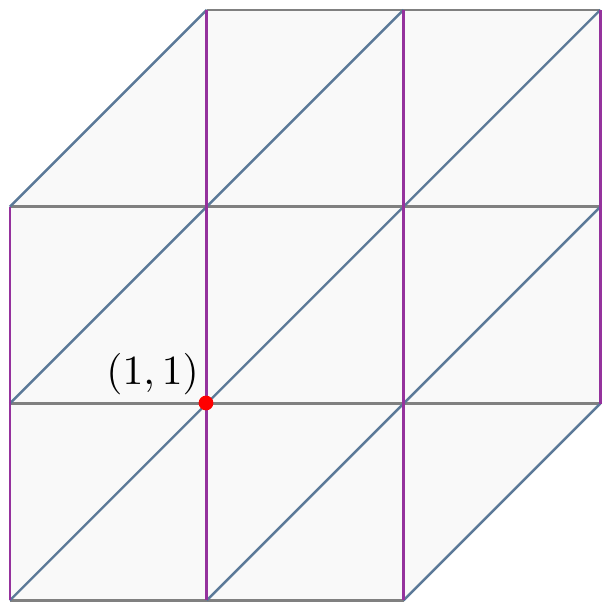}
	\end{center} 
	\vspace{-0.2cm}
	\caption{Support of the Courant hat function $\calB_{\boldsymbol 1}$, it corresponds to  $\St(\nu)$ where $\nu$ is the vertex $(1,1)$ (left), support of the box splines $\calB_{(2,1,1)}$ (center) and  $\calB_{(2,2,1)}$ (right).}\label{fig:courant}
\end{figure}
The coordinates $n_i$ of $\bn$ denote the number of convolutions of
$\calB_{\boldsymbol 1}$ along the directions $\bm e_i$. 
The support of the box spline $\calB_\bn$ is the zonotope in $\R^2$ formed
by the Minkowski sum of the direction vectors $\bm e_i$ taken $n_i$
times, for $i=1,2,3$, respectively (see Figure \ref{fig:courant} for an example). 
Moreover, for every $\bn$, the box spline $\calB_{\bm n}$ is strictly positive for all $\bm x$ in the interior of its support, and zero otherwise \cite{LaiSchu}*{Theorem 12.2}. 

From the general theory of type-I box splines, it follows that the box spline
$\calB_\bn$ is independent of the order in which of the vectors $\bm e_i$ appear in the recursive construction of $\calB_\bn$ in Definition~\ref{def:convolution}. 
This result follows immediately form the formula for the Fourier transform of a type-I box spline \cite{LaiSchu}*{Theorem 12.6}.

For any $\bn$, the box spline $\calB_\bn$ is a piecewise polynomial function on three directional triangulation $G$, and each polynomial is of total degree $n=|\bn| -2$, where $|\bn|=n_1+n_2+n_3$. 
It is also well-known that $\calB_\bn$ is a $C^r$-continuous function on $\R^2$, where $r=\min_{i=1,2,3}\bigl\{ |\bn|-n_i-2 \bigr\}$ \cite{LaiSchu}*{Theorem 2.4}.

There is a rich literature on type-I box splines, a detailed
construction and the proof of structural and smoothness properties can be found for
instance in \cite{LaiSchu}*{Chapter 12} or \cite{Boor}.
In particular, it is known that each convolution along a direction $\bm
e_i$ increases the continuity of the box spline with respect to differentiation in
that direction by one \cite{LaiSchu}*{Theorem 12.3}. Thus, following the notation
introduced in Section \ref{sec:prelim}, $\bd = (n_2+n_3 -2,\, n_1+n_3 -2,\,  n_1+n_2-2)$ then 
\begin{equation}
	\label{eq:Bn}
	\calB_\bn \in \SSS^\bd(G,\calP_n)\ .
\end{equation}
We denote by $B_\bn(G)$ the set of integer \emph{translates of the box spline} $\calB_\bn$, which is defined as the set 
\begin{equation}\label{eq:translates}
	B_\bn(G) =\bigl\{\calB_\bn(\, \cdot - \bm v)\, \colon \bm v\in\Z^2\bigr\}\ .
\end{equation}
The set $B_\bn(G)$ is also called the set of \emph{shifted box splines} associated to the direction vector $\bn$. 
\begin{remark}\label{rem:lind} 
	Notice that the translates in $B_\bn(G)$ have distinct support, it is the zonotope which is the support of $\calB_{\bm n}$ shifted by $\bm v$. 
	In fact, the set $B_\bn(G)$ is (globally) linearly independent \cite{LaiSchu}*{Theorem 12.19} i.e., if 
	\[\sum_{\bm v\in \ZZ^2} a_{\bm v}\calB_{\bm n}(\bm x -\bm v)=0, \text{ for all }  \bm x\in\RR^2\]
	then $a_{\bm v}=0$ for all $\bm v\in \ZZ^2$.
	Furthermore, it has been shown that the translates in $B_\bn(G)$ are also locally linearly independent i.e., if $A$ is an open set,  then the shifted box splines 
	\[
	\bigl\{\calB_\bn(\cdot-\bm v)\colon \supp\bigl(\calB_\bn(\cdot-\bm v)\bigr)\cap A\neq \emptyset\bigr\}
	\]
	are linearly independent~\cites{DahM85,jia85}.
	Here, $\supp(\beta)$ denotes the support of the function $\beta$ (Definition \ref{def:support}).
\end{remark}
We now introduce the definition of the set of active box splines on a given multicell domain $M$ in the type-I triangulation $G$. 
\begin{definition}\label{def:active}
	If $M \subseteq G$ is a multicell domain, we define
	\begin{align*}
		\Sup_{\bm n}(M) & =\bigl\{\bm v\in \Z^2\colon \supp\bigl(\calB_{\bm n}(\cdot -\bm v)\bigr)\cap M^*\neq \emptyset  \bigr\}, \\
		\intertext{and the set of \emph{active box splines} $\calB_\bn$ on $M$ by}
		B_\bn(M) & = \bigl\{\calB_{\bm n}(\cdot -\bm v)\bigr|_{M} \colon \bm v\in\Sup_{\bm n}(M)\bigr\},
	\end{align*}
	where $M^*$ is the closure of $M$ in $\R^2$ as defined in Equation \eqref{eq:M*}.
\end{definition}
In particular, if the multicell domain $M=\hat{\triangle}$, for a triangle $\triangle\in T$, then
$B_\bn(\hat\triangle)$ is the set of 
translates $\calB_\bn(\,\cdot -\bm v)\in B_n(G)$ whose support contains $\triangle$.
The number of elements in $B_\bn(\hat\triangle)$ is given by 
\begin{equation} \label{eq:phi}
	\phi(\bn) =  n_1n_2+n_1n_3+n_2n_3,
\end{equation}
where $\bn=(n_1,n_2,n_3)$. 
\begin{definition}
	If $\calB_{\bm n}(\cdot -\bm v)$ is the translate of the type-I box spline $\calB_{\bm n}$ by $\bm v\in\ZZ^2$, then we take $(1,1)-\bm v$ as the point of reference of $\supp\bigl(\calB_{\bm 2}(\cdot -\bm v)\bigr)$ in the lattice. For a triangle $\triangle\in T$, we define the \emph{1-ring neighbourhood} of $\triangle$ as the set of \emph{reference lattice points} $(1,1)-\bm v$ such that $\sup\bigl(\calB_{\bm n}(\cdot -\bm v)\bigr)\cap \triangle \neq \emptyset$.
\end{definition}
For instance, if $\bm 2 =(2,2,2)$ then the elements in $B_{\bm 2}(\hat\triangle)$
are the translates $\calB_{\bm 2}(\cdot -\bm v)$ associated to the $\phi(\bn)=12$ lattice points
in a 1-ring neighbourhood of $\triangle$ in the grid $G$, which are shown in Figure~\ref{fig:onering}.
\begin{figure}[ht]
	\centering
	\includegraphics[width=4.5cm]{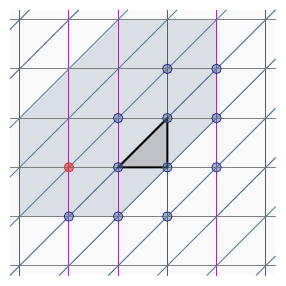}
	\caption{The 12 lattice points correspond to the 1-ring of a triangle $\triangle\in T$ associated to the box spline $\calB_{\bm 2}=\calB_{(2,2,2)}$. Each of these points is the reference point $(1,1)-\bm v$ for the translate $\calB_{\bm 2}(\cdot -\bm v)$ such that $\triangle\subseteq \supp(\calB_{\bm 2}(\cdot -\bm v))$.}\label{fig:onering}
\end{figure}

Notice that to any type-I box spline $\calB_{\bm n}$ and a triangle  if $\triangle\in T$ we can associate a linear space of polynomials. Namely, extending by linearity, we can take $\calV_\bn|_\triangle$ as the space generated by the restriction of the active box splines $\calB_{\bm n}$ to the triangle $\triangle$ i.e.,  
\begin{equation}\label{eq:Vn}
	\calV_{\bm n}|_\triangle = \spanset B_\bn(\hat\triangle)\ .
\end{equation}
By Equations \eqref{eq:Vn} and \eqref{eq:Bn}, we see that $\calV_\bn|_\triangle$ is a linear subspace of $\calP_n\subseteq \RR[x,y]$, for any $\triangle\in T$ and $n=n_1+n_2+n_3-2$.

We now prove that for any $\bn\in\ZZ^3$, the subspace $\calV_{\bm n}|_\triangle\subseteq \calP_n$ is independent of the choice of the triangle $\triangle\in T$. 
The proof of this result is a generalization of \cite{vmj2015}*{Proposition 27}, where we consider the case $\bm 2=(2,2,2)$ and the box spline $\calB_{\bm 2}$.
\begin{proposition}\label{prop:spaceVn}
	Let $\bn=(n_1,n_2,n_3)\in\Z^3_{\geq 1}$ and $f \in \calP_n$, with $n=n_1+n_2+n_3-2$. Then,
	$\calV_\bn|_{\triangle}=\calV_\bn|_{\triangle'}$ for any pair of triangles $\triangle$ and $\triangle'$ in $T$.
\end{proposition}
\begin{proof}
	Let $\triangle$ and $\triangle'$ be two triangles in the three-directional grid $G$. Denote by $G^0$ the triangulation of $\R^2$ obtained by the lines parallel to the vectors $\bm e_1'=(a,0)$, $\bm e_2'=(0,a)$ and $\bm e_1'=(a,a)$, for a fixed number $a\in \Z$. Notice that for any $a\in \Z$, we can see the grid $G$ as a refinement of a grid $G^0$. Denote by $\triangle_a$ the triangle in $G^0$ with vertices at $(0,0)$, $(0,a)$ and $(a,a)$. In particular, let us take $\ell\in \Z_+$ and $a=2^\ell$ in such a way that the translate $\tilde\triangle=\triangle_a - \bm v$ of $\triangle_a$ by a vector $\bm v\in \Z^2$ contains both $\triangle$ and $\triangle'$. Then, $B_\bn( 1/2^\ell\, \cdot )$ is the correspondent box spline associated to $\bn$ in the grid $G^0$.
	By the refinement equation for box splines \cite{LaiSchu}*{Theorem 12.9}, there exists a finite sequence $\{c_{\bm v}\}_{\bm v\in\Z^2}$ such that 
	\begin{equation}\label{eq:refeq}
		B_\bn( 1/2^\ell\, \cdot )=\sum_{\bm v\in\Z^2}c_{\bm v} B_\bn(\cdot - \bm v)\ .
	\end{equation}
	Let us denote by $\calV_\bn^0|_{\tilde\triangle}$ the span of the box spline translates $\calB_\bn(1/2^\ell\cdot-\bm v)$ restricted to $\tilde\triangle$ in $G^0$. By the symmetry of the box splines supports, the number $\phi(\bm n)$ of active translates on a triangle is independent of the grid and of the given triangle in the grid.  Furthermore, these translates are linearly independent  (see Remark \ref{rem:lind}). Thus,
	\begin{equation}\label{eq:dimV}
		\dim \calV_\bn^0|_{\tilde\triangle} = \dim  \calV_\bn|_{\triangle}= \dim  \calV_\bn|_{\triangle'}.
	\end{equation}
	Taking the restrictions to $\triangle$ and $\triangle'$, Equation \eqref{eq:refeq} implies $ \bigl(\calV_\bn^0|_{\tilde\triangle}\bigr)|_\triangle\subseteq \calV_\bn|_\triangle$ and 
	$ \bigl(\calV_\bn^0|_{\tilde\triangle}\bigr)|_{\triangle'}\subseteq \calV_\bn|_{\triangle'}$. 
	
	Since $\calV_\bn^0|_{\tilde\triangle}$ is a polynomial subspace of $\calP_n$, we have $\calV_\bn^0|_{\tilde\triangle}=(\calV_\bn^0|_{\tilde\triangle})|_\triangle =(\calV_\bn^0|_{\tilde\triangle})|_{\triangle'}$. In particular, both $(\calV_\bn^0|_{\tilde\triangle})|_\triangle $ and $(\calV_\bn^0|_{\tilde\triangle})|_{\triangle'}$  have the same dimension as $\calV_\bn^0|_{\tilde\triangle}$. The statement follows by applying Equation \eqref{eq:dimV}.
\end{proof}
Proposition \ref{prop:spaceVn} implies that for a given $\bm n\in\Z^3_{\geq 1}$ the restrictions of the translates of the box spline $\calB_{\bm n}$ to a triangle $\triangle\in T$ define a polynomial space which is independent of $\triangle$. From now on, we will denote such polynomial space as $\calV_\bn$. 

\begin{remark}\label{rem:span}
	For any $\triangle\in T$ an $\bn\in\ZZ_{\ge 1}^3$, Remark \ref{rem:lind} implies that $B_\bn(\hat\triangle)\subseteq \calP_n$ is a linearly independent set.
	Here $\calP_n$ is as before, the set of bivariate polynomials of degree at most $n=n_1+n_2+n_3-2$, and $\bn =(n_1,n_2,n_3)$. 
	The number of elements in $B_\bn(\hat\triangle)=\phi(\bn)$ (Equation \eqref{eq:phi}) is the dimension of the space of polynomials $\dim \calV_{\bn}$ associated to the type-I box $\calB_{\bm n}$. Thus, $\dim(\calV_{\bn})=n_1n_2+ n_1n_3+n_2n_3\leq \binom{n_1+n_2+n_3}{2}= \dim \calP_n$. 
	
	For $n_i\geq 1$, equality holds only for $\bn=\bm 1= (1,1,1)$, which corresponds to polynomial space $\calV_{\bm 1}$ associated to the Courant hat function $\calB_{\bm 1}$ in Definition \ref{def:convolution}.
	For any other $\bn\in \ZZ^3$ and the corresponding box spline $\calB_{\bm n}$, the polynomial space $\calV_{\bn}$ is a proper subspace of $\calP_\bn$.
\end{remark}	

\section{Characterization of box spline spaces}\label{sec:ccl}

Let $\triangle\in T$, $\bn\in \Z_{\geq 1}^3$, and consider $\calV_\bn= \spanset B_\bn(\hat\triangle)$. As observed in Remark \ref{rem:span}, the set $B_\bn (\hat\triangle)$ of the box spline translates with support on $\triangle$ is linearly independent, and hence the restriction of a polynomial $f\in \calV_\bn$ to $\triangle$ has a unique representation
\begin{equation}\label{eq:linearc}
	f|_\triangle(\bbx) = \sum_{\beta\in B_\bn(\hat\triangle)} \lambda_\triangle^\beta (f|_\triangle)\beta(\bbx), \quad \bbx\in\overline\triangle,
\end{equation}
for coefficients $\lambda_\triangle^\beta(f|_\triangle)\in\R$.

\begin{lemma}\label{lem:decomp1}
	Let $\bm n=(n_1,n_2,n_3)\in\Z^3_{\geq 1}$, $f$ and $f'$ in $\calV_\bn$, and  $M=\hat\triangle\cup\hat\triangle'$, such that  $\triangle, \triangle'\in T$ and
	$\overline\triangle\cap\overline\triangle'\in E_i$ for some $i\in\{1,2,3\}$.
	
	If $\bigl(f|_\triangle, f'|_{\triangle'}\bigr)\in\Splsp^{\bd}(M,\calV_\bn)$,  then there exist polynomials $g, h\in \calV_{\bm n}$ such that 
	\begin{equation}\label{eq:fdecomp}
		\bigl(f|_\triangle, f'|_{\triangle'}\bigr)=\bigl(g|_\triangle, f'|_{\triangle'}\bigr)+\bigl(h|_\triangle, 0|_{\triangle'}\bigr),
	\end{equation}
	with $\lambda_\triangle^\beta\bigl(g|_\triangle\bigr) =
	\lambda_{\triangle'}^\beta\bigl(f'|_{\triangle'}\bigr)$ for every  $\beta\in B_\bn(\hat\triangle)\cap B_\bn(\hat\triangle')$.
\end{lemma}
\begin{proof}
	Define
	\[g|_\triangle=\sum_{\beta\in B(\hat\triangle)\cap B(\hat\triangle')}\lambda_\triangle^\beta(f'|_\triangle) \beta|_\triangle + \sum_{\beta\in B(\hat\triangle)\setminus B(\hat\triangle')}0\cdot \beta|_\triangle.\]
	Then $g|_\triangle\in \calV_{\bm n}|_\triangle$, and extending by linearity we can see $g|_\triangle$ as the restriction to $\triangle$ of a polynomial $g$ in $\calV_{\bm n}\subseteq \R[x,y]$. By construction, the pair $(g|_\triangle,f'|_{\triangle'})$ satisfies the required condition, and taking $h=f-g$ we obtain Equation \eqref{eq:fdecomp}.
\end{proof}
\begin{lemma}\label{lemma:ccl}
	Let $\bm n=(n_1,n_2,n_3)\in\Z^3_{\geq 1}$, $f$ and $f'$ in $\calV_\bn$, and  $\triangle, \triangle'\in T$, such that 
	$\overline\triangle\cap\overline\triangle'\in E_i$ for some $i\in\{1,2,3\}$. Then, for $M=\hat\triangle\cup\hat\triangle'$ the following two statements are equivalent:
	\begin{enumerate}
		\item[(i)] The pair $(f|_\triangle, f'|_{\triangle'})\in \Splsp^{\bd}(M,\calV_{\bn})$, with $\bd=(n_2+n_3-2,n_1+n_3-2,n_1+n_2-2)$.
		\item[(ii)] For every $\beta\in B_\bn(\hat\triangle)\cap B_\bn(\hat\triangle')$, it holds $\lambda_\triangle^\beta(f|_\triangle) =
		\lambda_{\triangle'}^\beta(f'|_{\triangle'})$.
	\end{enumerate}
\end{lemma}
\begin{proof} 
	(ii) $\Rightarrow$ (i) Let $(f,f')\in \calV_\bn$ be a pair of polynomials satisfying (ii). Then
	\begin{equation}
		(f|_\triangle,f'|_{\triangle'}) = \sum_{\beta\in B_\bn(M)}\lambda_\triangle^\beta(f|_{\triangle})\beta|_{\triangle\cup \triangle'} + ( h|_{\triangle}, 0|_\triangle)  +(0|_\triangle, h'|_{\triangle'}),\label{eq:decomp2}
	\end{equation}
	where
	\[h|_{\triangle} =\hspace{-0.5cm} \sum_{\beta\in B_\bn(\hat\triangle)\setminus B_\bn(\hat \triangle')} \hspace{-0.5cm}\lambda_\triangle^\beta (f|_{\triangle})\beta|_\triangle\ , 
	\text{\, and \; } 
	h'|_{\triangle'} =\hspace{-0.5cm}\sum_{\beta\in B_\bn(\hat\triangle')\setminus B_\bn(\hat \triangle)} \hspace{-0.5cm} \lambda_\triangle^\beta (f'|_{\triangle'})\beta|_\triangle.\]
	
	Since $\beta\in \Splsp^{\bd}(M,\calV_\bn)$
	for every $\beta\in B_\bn(M)$, then the first term in Equation \eqref{eq:decomp2} is in $\Splsp^{\bd}(M,\calV_\bn)$. Now, if $\beta\in B(\hat{\triangle})\setminus B(\hat{\triangle'})$ then $(\beta|_\triangle, 0|_{\triangle'})=\beta|_{M}$, and similarly for  $\beta\in B(\hat{\triangle'})\setminus B(\hat{\triangle})$. Hence, the last two terms in Equation \eqref{eq:decomp2} are also in $\Splsp^{\bd}(M,\calV_\bn)$, and (i) follows.
	
	\medskip
	
	(i) $\Rightarrow$ (ii) By Lemma \ref{lem:decomp1} it is enough to consider $(f|_\triangle, 0|_{\triangle'})\in \Splsp^{\bd}(M,\calV_{\bm n})$. We need to show that $\lambda_\triangle^\beta(f|_\triangle)=0$ for every $\beta\in B(\hat{\triangle})\cap B(\hat{\triangle'})$. To simplify, by an abuse of notation we will denote $\Sup_{\bm n}(\hat{\triangle})\cap \Sup_{\bm n}(\hat{\triangle'})$ simply by $\Sup_{\bm n}(\hat \varepsilon_3 )$.
	
	\medskip
	
	Let $\overline \triangle\cap \overline\triangle'=\varepsilon_i\in E_i$, and denote by $b_{i}$ the barycentric coordinate relative to $\triangle$ which vanishes at the edge $\varepsilon_i$.
	Let $n=n_1+n_2+n_3$.
	Then  $(f|_\triangle, 0|_{\triangle'})\in  \Splsp^{\bd}(M,\calV_{\bm n})$ if and only if the polynomial $b_i^{n-n_i-1}$ divides  $f$, see \cite{bil}*{Lemma 2.2}. Notice that the latter condition is satisfied if and only if the Bernstein-B\'ezier coefficients $c_{j k \ell}$ of $f|_{\triangle}$ are zero for every $0\leq j\leq n-n_1-2$. 
	
	\medskip
	
	Without loss of generality we can assume that $i=3$ and $\triangle$ is the triangle with vertices at $(0,0), (1,0)$, and $(1,1)$. Thus, $b_3=x-y$. We will prove that if $f\in \calV_{\bm n}$ and $b_3^{n-n_3-1}$ divides  $f$ then $\lambda_\triangle^\beta(f|_\triangle)=0$ for every $\beta\in B(\hat{\triangle})\cap B(\hat \triangle')$.
	We proceed by induction on $n$, with $n_i\geq 1$.
	
	\medskip
	
	The induction base is $n=3$, with $\bn =(1,1,1)$, $n_3=1$ and $n-n_3-1=1$. We have
	\begin{align*}
		B_{\bn} (\hat\triangle\cup \hat\triangle')&=
		\{\beta_j\}_{j=0}^3,
	\end{align*}
	with $\beta_0=\calB_{\bn}$, and $\beta_i= \calB_{\bn}(\cdot +\bm  e_i)$ for $i=1,2,3$. Thus, the only translates with support on $\triangle$ are $\beta_0,\beta_2$ and $\beta_3$. Since $\beta_0=b_1$, $\beta_2=b_3$, and $\beta_3=b_2$, then 
	\begin{align*}
		f|_\triangle& = \lambda^{\beta_0}_\triangle (f|_\triangle) b_1|_\triangle + \lambda^{\beta_2}_\triangle (f|_\triangle) b_3|_\triangle +  \lambda^{\beta_3}_\triangle (f|_\triangle) b_2|_\triangle.
	\end{align*} 
	Since the polynomials $b_1,b_2,b_3$ are linearly independent, and by hypothesis $b_1|f$, it follows $\lambda^{\beta_i}_\triangle (f|_\triangle)=0$ for $i=2,3$. This proves the statement for the case $n=3$. 
	
	Let us assume that the result is true for every $m\leq n$ for some  $n\geq 3$. Let $\bm n =(n_1,n_2,n_3)$ with  $n+1 = n_1+n_2+n_3$.
	We consider three cases, in the first two cases only one of the indices $n_i$ is $\geq 2$ and in the last case at least two of the indexes are $\geq 2$.
	\begin{description}
		\item[\emph{Case 1.}] 
		Suppose $n_1\geq 2$, and $n_2=n_3 = 1$. 
		
		\noindent We have $(x-y)^{n_1}|f$ and $D_{\bm e_1}f\equiv 0 \mod (x-y)^{n_1-1}$. Notice that Equation \eqref{eq:linearc} may be rewritten as
		\begin{equation}\label{eq:rewritef}
			f|_\triangle=\sum_{\bm v\in\Z^2} \lambda_{\bm v}(f)\calB_{\bm n}(\cdot -\bm v)\bigr|_\triangle, 
		\end{equation}
		where $\lambda_{\bm v}=0$ for every $\bm v\notin \Sup_{\bm n}(\triangle)$, and $\lambda_{\bm v}(f)=\lambda_\triangle^\beta(f|_\triangle)$ for $\beta\in B(\hat{\triangle})$ and $\beta=\calB_{\bm n}(\cdot -\bm v)$.

		\noindent Moreover, we know that for any linear combination of box splines
		\begin{align}
			D_{\bm e_i} \sum_{\bm v\in\Z^2} a_{\bm v}\calB_{\bm n}(\cdot -\bm v)& = 
			\sum_{\bm v\in\Z^2} a_{\bm v}\bigl(\calB_{\bm m}(\cdot -\bm v)- \calB_{\bm m}(\cdot -\bm v-\bm e_1)\bigr) \nonumber\\
			& = \sum_{\bm v\in\Z^2}(a_{\bm v+\bm e_i}- a_{\bm v}) \calB_{\bm m}(\cdot -\bm v +\bm e_i), \label{eq:derivative}
		\end{align} 
		for $a_{\bm v}\in\R$ for every $\bm v\in\Z^2$, and $\bm m =(n_1-1,n_2, n_3)$, see \cite{LaiSchu}*{Lemma 12.3}.
		
		\noindent Thus, 
		\begin{align*}
			D_{\bm e_1} \sum_{\bm v\in\Z^2} \lambda_{\bm v}(f)\calB_{\bm n}(\cdot -\bm v)\bigr|_\triangle &= \sum_{\bm v\in\Z^2}\bigl(\lambda_{\bm v+\bm e_1}(f)- \lambda_{\bm v}(f)\bigr) \calB_{\bm m}(\cdot -\bm v +\bm e_1)\bigr|_\triangle\\
			&\equiv 0 \mod (x-y)^{n_1-1},
		\end{align*} 
		where $\bm m=(n_1-1,1,1)$.
		By induction hypothesis $\lambda_{\bm v+\bm e_1}(f)= \lambda_{\bm v}(f)$ for every $\bm v\in \Sup_{\bm m}(\hat\varepsilon_3)$. 
		
		\noindent Notice that, since $\bm n =(n_1,1,1)$ then the elements $\bm v\in \Sup_{\bm v}(\hat\triangle)$ are either a multiple $t\bm e_1$ of $\bm e_1$ or of the form $\bm e_2+ t\bm e_1$, for $t\in \Z$. Thus, we have
		\[
		f|_\triangle= \lambda_{\bm e_1}(f)\sum_{t\in\Z}\calB_{\bm n}(\cdot -t\bm e_1)|_\triangle +\lambda_{\bm e_2}(f)\sum_{t\in\Z} \calB_{\bm n}(\cdot -\bm e_2 +t\bm e_1)|_\triangle , 
		\]
		for constants $\lambda_{\bm e_i}(f)\in\R$, for $i=1,2$.
		Moreover, $\calB_{\bm n}(\cdot -t\bm e_1)|_{\overline\varepsilon_1}=0$ for every $t$. In particular, they are zero at the vertex $(0,0)$. But $f(0,0)=0$, and 
		\[1=\sum_{\bm v\in\Z^2}\calB_{\bm n}(\cdot -\bm v)=\sum_{t\in\Z} \calB_{\bm n}(\cdot -t\bm e_1) + \calB_{\bm n}(\cdot-\bm e_2+t\bm e_1).\]
		Then $f|_{(0,0)}=\lambda_{\bm e_2}(f)\sum_{t\in\Z} \calB_{\bm n}(\cdot -\bm e_2 +t\bm e_1)|_{(0,0)}=\lambda_{\bm e_2}(f)=0$. 
		We obtain $\lambda_{\bm e_1}(f)=0$ by considering the restriction at the vertex $(1,1)$.
		
		\medskip
		
		\item[\emph{Case 2.}]
		Suppose $n_3\geq 2$ and $n_1=n_2=1$.
		
		By hypothesis $(x-y)|f$, and so $D_{\bm e_3}f\equiv 0 \mod (x-y)$.
		Following the same argument as above, $\lambda_{\bm v+\bm e_3}(f)=\lambda_{\bm v}(f)$ for every $\bm v\in \Sup_{\bm n}(\hat\varepsilon_3)$.
		Thus, there is a constant $\lambda_{\bm e_3}$ which is equal to $\lambda_{\bm v}(f)$ for every $\bm v\in \Sup_{\bm n}(\hat\varepsilon_3)$, and
		\[
		f|_{\varepsilon_3}= \lambda_{\bm e_3}(f)\sum_{\bm v\in\Z^2}\calB_{\bm n}(\cdot-\bm v)\bigr|_{\varepsilon_3}.
		\] 
		Since $f|_{\varepsilon_3}=0$,  and $\sum_{\bm v\in\Z^2}\calB_{\bm n}(\cdot-\bm v)\bigr|_{\varepsilon_3}=1$, it follows $\lambda_{\bm e_3}(f)=0$.
		
		\medskip
		
		\item[\emph{Case 3.}] At least two of the indices $n_i$ are $\geq 2$, say $n_1,n_2\geq 2$. 
		By hypothesis $(x-y)^{n_1+n_2-1}|f$, and so $D_{e_i} f \equiv 0 \mod (x-y)^{n_1+n_2-2}$ for $i=1,2$. Similarly as before, rewrite $f|_\triangle$ as in Equation \eqref{eq:rewritef}, and consider Equation \eqref{eq:derivative} with $i=1$. Thus,
		\begin{align*}
			D_{\bm e_1} \sum_{\bm v\in\Z^2} \lambda_{\bm v}(f)\calB_{\bm n}(\cdot -\bm v)\bigr|_\triangle &= \sum_{\bm v\in\Z^2}\bigl(\lambda_{\bm v+\bm e_1}(f)- \lambda_{\bm v}(f)\bigr) \calB_{\bm m}(\cdot -\bm v +\bm e_1)\bigr|_\triangle\\
			&\equiv 0 \mod (x-y)^{n_1+n_2-2}.
		\end{align*} 
		By induction hypothesis $\lambda_{\bm v+\bm e_1}(f)= \lambda_{\bm v}(f)$ for every $\bm v\in \Sup_{\bm m}(\hat\varepsilon_3)$. 
		Similarly, by considering $D_{\bm e_2}f$ we get $\lambda_{\bm v+\bm e_2}(f)=\lambda_{\bm v}(f)$ for every $\bm v\in \Sup_{\bm m'}(\hat\varepsilon_3)$, where $\bm m'=(n_1,n_2-1, n_3)$.
		This implies that the coefficients $\lambda_{\bm v}(f)$ are equal to a constant $\lambda$, for every $\bm v \in\Sup_{\bm n}(\hat\varepsilon_3)$.
		
		\noindent Therefore, $f|_{\varepsilon_3}= \lambda\sum_{\bm v\in\Z^2}\calB_{\bm n}(\cdot -\bm v)|_{\varepsilon_3}$. But $f|_{\varepsilon_3}=0$, and $\sum_{\bm v\in\Z^2}\calB_{\bm n}(\cdot -\bm v)=1$, in particular when taking the restriction to $\varepsilon_3$. Hence, $0=\lambda=\lambda_\triangle^\beta(f)$, for every $\beta\in B(\hat\triangle)\cap B(\hat{\triangle'})$.
	\end{description}	
	An analogous proof applies when the edge of intersection between $\overline{\triangle}$ and $\overline{\triangle'}$ is parallel to one of the other two vectors $\bm e_1$ or $\bm e_2$.
\end{proof}

The edge-contact property plays a fundamental role for the construction of hierarchical spline spaces, we present that construction for any type-I box spline in Section \ref{sec:H}. 
The following examples illustrate that importance of the properties of type-I box splines used to prove Lemma \ref{lemma:ccl}, we show that not every spline space possesses those properties.
\begin{example}\label{ex:two}
	Let $M=\hat\triangle\cup\hat\triangle'$ be the multicell domain in Figure~\ref{fig:twotr}. 
	\begin{figure}[ht]
		\centering
		\includegraphics[width=5cm]{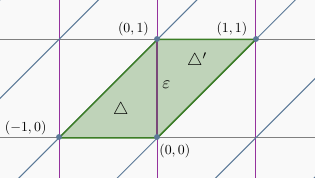}
		\caption{The triangles $\triangle$ and $\triangle'$ in $G$ define the multicell domain $M$ in Example \ref{ex:two}, the triangles share the edge $\varepsilon\in E_2$ that is parallel to the directional vector $\bm e_2$.}\label{fig:twotr}
	\end{figure}
	
	Take the set $B=\{f_1, \dots, f_4\}$ of spline functions $f_i$ on $M$ defined as 
	\[	f_1 = \bigl(y^2 + 3x^2 + 4x,\, y^2 + x \bigr)\ , \; 	f_2 = \bigl(4y^2,\, 4y^2\bigr)\ ,\]
	\[ f_3 = \bigl(x^2,\, 0\bigr)\ , \text{ and }  f_4 = \bigl(0,\, x^2\bigr)\ .\]
	Each pair of polynomials $f_i$ is defined on $\triangle$ and $\triangle'$, respectively. Then, $B$ is contained in the space of splines $ \SSS^{\bm 0}(M,\calQ)$, where $\bm 0=(0,0,0)$, and $\calQ\subseteq \calP_2$ is the linear subspace of polynomial spanned by $x$, $x^2$, and $y^2$.
	
	We have
	\begin{align*}
		\calQ|_{\triangle} \;  = B|_{\triangle} \;  &= \bigl\{y^2 + 3x^2 + 4x,\, 4y^2,\, x^2\bigr\}\bigl|_{\triangle} \ , \text{ and }\\
		\calQ|_{\triangle'} = B|_{\triangle'} &= \bigl\{y^2 + x,\, 4y^2,\, x^2\bigr\}\bigl|_{\triangle'}.
	\end{align*}
	Then, $g=\bigl(5\cdot(y^2 + 3x^2 + 4x) + 4y^2,\,  y^2 + x + 2 \cdot (4y^2)\bigr)\in \SSS^{\bm 0}(M,\calQ)$, but $g$ is not an element in $\spanset(B)$. \hfill $\Diamond$
\end{example}

The following example illustrates the importance of the spline space $\SSS^\bd (M, \calV_\bn)$ (Definition \ref{def:edge-splines}) in the proof of Lemma \ref{lemma:ccl}.

\begin{example}
	Let $M=\hat\triangle\cup\hat\triangle'$ be the multicell domain defined from the triangles $\triangle$ and $\triangle'$ such that $\triangle\cap\triangle'=\varepsilon\in E_3$ in Figure \ref{fig:example2}.
	Take $\bn=(2,1,1)$, and consider the linear space $\calV_\bn= \spanset B_\bn(\hat{\triangle})$.  
	The support and Bernstein coefficients of the box spline $\calB_{(2,1,1)}$ are displayed in Figure \ref{fig:example2}. The generators of $\calV_{(2,1,1)}$ can easily be described by restricting the translates $\calB_{(2,1,1)}(\cdot -\bm v)\in B_\bn(M)$ to $\triangle'$ and using the Bernstein coefficients of $\calB_{(2,1,1)}$. Namely, $\calV_{(2,1,1)}$ is generated by 
	\[f_1 = 2(x-y)y +y^2\ ; \;		f_2 = 2(1-x)y +y^2\ ;\]
	\[f_3 = (x-y)^2\ ; \]
	\[f_4= (1-x)^2+4(1-x)(x-y) + (x-y)^2 + 2(1-x)y +2(x-y)y\ ;\] 
	\[\text{and } \; f_5= (1-x)^2\ . \]
	In particular, \[f= 4(x-y)y+4(1-x)(x-y)+(x-y)^2= f_1 - f_2 + f_4 - f_5 \in \calV_{(2,1,1)} \ .\]
	Let us notice that $g=\bigl(0|_{\triangle},f|_{\triangle'}\bigr)\in \SSS^{\bm 0}(M,\calV_\bn) $, but $g\notin \spanset \calB_\bn (M)$. In fact,  for $\beta(\cdot)=\calB_\bn(\cdot)$ we have $\beta|_{\triangle'}=f_1$, and $\lambda^\beta_{\triangle'}(f|_{\triangle'})=1$, but $\lambda^\beta_{\triangle}(0|_{\triangle})=0$. 
	
	On the other hand, $\varepsilon\in E_3$ i.e., $\varepsilon$ is an edge parallel to the directional vector $\bm e_3$,  and $g$ is not a $C^1$-continuous spline on $M$. Thus,   $g \notin\SSS^\bd(M,\calV_\bn)$ for $\bd=(0,1,1)$.\hfill $\Diamond$
	
\end{example}

\begin{figure}
	\centering
	\includegraphics[width=5cm]{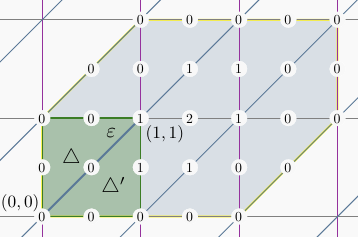}
	\caption{Bernstein coefficients of the box spline $2\calB_{(2,1,1)}$.}\label{fig:example2}
\end{figure}

\begin{definition}\label{def:admissible}
	A multicell domain $M$ is \emph{admissible} with respect to a type-I box spline
	$\calB_\bn$ if the support of any $\beta \in B_\bn(M)$ is a
	connected set, and there exist no over-concave vertices (Definition \ref{def:oconcave}), and no kissing triangles in $M$ (Definition \ref{def:kissing}).
\end{definition}
%\note[inline]{is this the same as graph-condition? which is one is better?}

%\cmb{Expain: why we need these conditions. Why we need SIC (in other case we would need decoupling), why the rest, to satisfy \ref{prop::adm1}.
%Connection o Hierarchical construction for Section~\ref{sec:H} (actually we will need it for the coarsest grid $\Omega_0$only)
% }
In view of Lemma~\ref{lemma:ccl} we arrive to the following completeness result. 

\begin{corollary}[Completeness of type-I box splines]\label{cor:compl}
	If $M\subseteq G$ is an admissible domain, and $\bm n=(n_1,n_2,n_3)\in\ZZ^3_{\ge 1}$, then the generating set $B_\bn(M)$ is complete for
	$\Splsp^{\bd(\bn)}(M,\calV_{\bm n})$, where $\bd (\bn)=(n_2 +n_3 - 2, n_1 +n_3 - 2, n_1 +  n_2 - 2)$.% as in   Lemma~\ref{lemma:ccl},
\end{corollary}
%\begin{proof} 
%NOTE: we need to prove the CCL for both edge and vertex contact.
%\end{proof}

%An analogous result for box splines on tensor product grids 
%was proved in \cite{mokris}. It is also known that the completeness does not hold for box
%splines on Type II triangulations, \cite{}.
%The conjecture is verified computationally for all $\bn\leq (4,4,4)$ in Section~\ref{sec:exp}.

%\begin{condition}
%Maximal rank condition on the rank of the box splines supported on the edges.\note{N: why did we want this condition? what was it precisely?}
%\end{condition}

\section{Hierarchical type-I box splines}\label{sec:H}
In this section, to any given box spline $\calB_\bn$ we associate a space of \emph{special splines} defined on a hierarchical grid. We construct a hierarchical basis for such space and prove that this basis is complete under certain assumptions on the domain hierarchy.

For an integer $N\geq 0$, we recursively define a sequence of three-directional grids $G^1, G^2$, $...,G^N$ as follows. We take $G^1=G$ as the three-directional grid with vertices $\Z^2$ introduced in Section
\ref{sec:prelim}, Figure \ref{fig:grid}. The hierarchical grids 
\[
G^\ell \quad\text{for\; } \ell=2,\dots,  N
\]
are defined recursively, in such a way that $G^{\ell+1}$ is obtained from $G^\ell$ by one global, uniform dyadic refinement step. 
More precisely, the grid $G^{\ell+1}=\frac12G^\ell$ is the triangulation of $\R^2$ with vertices at points $1/2^\ell(k,k)$, obtained by drawing in the lines $x=k/2^{\ell}$,  $y=k/2^{\ell}$, and $x-y=k/2^{\ell}$, for all $k\in\ZZ$.
Thus, to construct $G^{\ell+1}$, every triangle in the grid $G^\ell$ is split into four smaller ones, as illustrated in Figure~\ref{fig:grids}. 
The index $\ell$ will be called the {\em level} of the grid, and the number $N$ specifies the number thereof. Each grid $G^\ell$ is a uniform three-directional grid, similarly as for the grid $G=G^1$ before, we consider the triangles in this grid as open sets in $\R^2$.

\begin{figure}[!ht]
	\centering
	\includegraphics[scale=0.8]{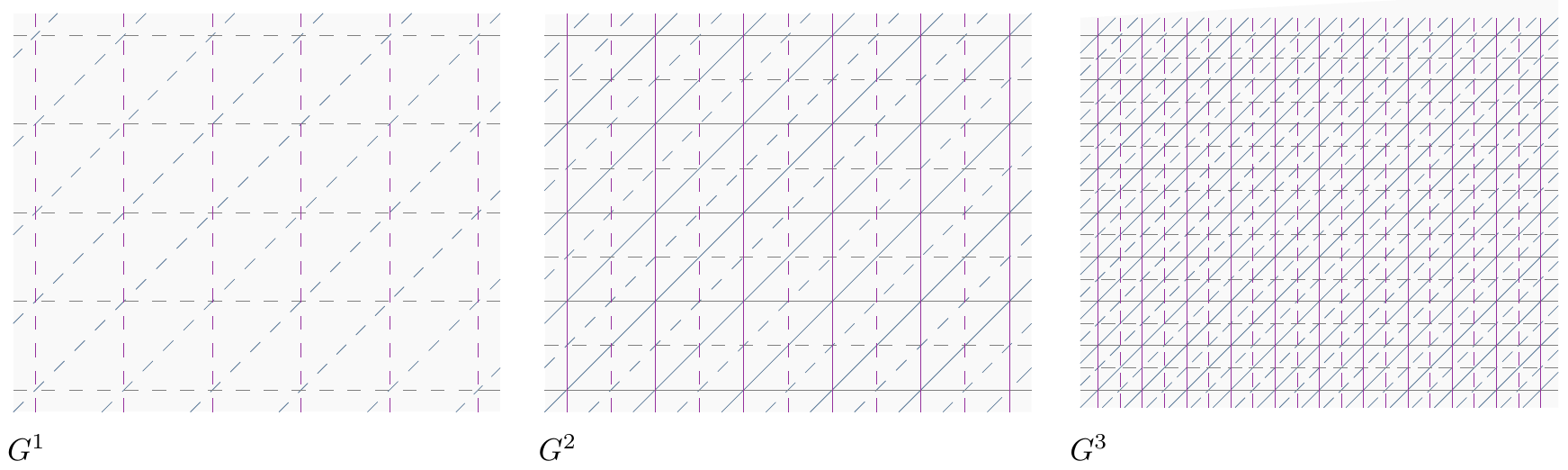}
	\caption{Three levels of three-directional hierarchical grids.}\label{fig:grids}
\end{figure}

Let $\Omega$ be a domain of $\R^2$ whose boundary $\partial\Omega$ is the union of edges from the grid $G^N$. We define a hierarchical multicell domain $H$ associated to the domain $\Omega$ as follows.

A \emph{nested sequence of subdomains} of $\Omega$ is defined as a collection of domains $\calM^\ell$ such that   
\[
\emptyset=\calM^0\subseteq\calM^1\subseteq\cdots\subseteq\calM^N=\Omega, 
\]
where $\displaystyle\calM^\ell= \bigcup_{\triangle\in M^\ell}\overline\triangle$, and $M^\ell\subseteq G^\ell$ is a multicell domain in the grid $G^\ell$, for  for each $\ell= 1,\dots, N$.

Thus, for each level $\ell$, the boundary $\partial\calM^\ell$ is a union of edges of the grid $G^\ell$. The difference between two successive subdomains, denoted $\calD^\ell$, is defined as the closure 
\[
\calD^\ell = \overline{\calM^\ell\setminus \calM^{\ell -1}}\ .
\]
The associated refined domain of level $D^\ell\subseteq G^\ell$ is defined as
$
D^\ell = T^\ell (\calD^\ell), 
$
where $T^\ell(\cdot)$ is the triangulation operator which restrict the grid $G^\ell$ to a given subset of the plane $\R^2$. More precisely, 
\[T^\ell(\calQ)=\bigl\{\hat\triangle\in G^\ell\colon \triangle\subset\calQ\bigr\}.\] 
The \emph{hierarchical multicell domain} $H$ associated to $\Omega$ is then the collection of triangles form all levels of the refinement area
\[
H= \bigcup_{\ell=1}^N D^\ell.\]
Using this notation, the domain $\Omega$ can be written as the union 
$
\Omega = \bigcup_{\triangle\in H} \overline\triangle\ .
$
\begin{definition}
	Let $H$ be a three-directional hierarchical multicell domain associated to a domain $\Omega\subseteq \R^2$, and let $\calB_n$ be a box spline, for some triple $\bn=(n_1,n_2,n_3)\in \Z^3_{\geq 1}$. We define $\PP(H, \calV_\bn)$ as the set of piecewise polynomial functions on $H$ associated to $\calB_n$ i.e., 
	\begin{equation*}
		\PP(H, \calV_\bn) = 
		\bigl\{
		f\in C^0(\Omega) \colon   f|_\triangle\in\calV_\bn|_\triangle \text{\, for each triangle } \triangle\in H
		\bigr\}\ .
	\end{equation*}
	If $\bd=(n_2+n_3-2,n_1+n_3-2,n_1+n_2-2)$, the \emph{hierarchical box spline space with edge smoothness $\bd$ on $H$} is defined as the set 
	%%\begin{equation}
	%\HH = S^{\bd(\bn)}(H)
	%\end{equation}
	%for a linear space of functions $\hat\calP_\bn$ (for some tuple ) on $\Omega$.
	%
	\begin{multline}\label{eq:HH}
		\SSS^\bd(H,\calV_\bn) = \bigl\{ f\in\PP(H,\calV_\bn) \colon
		f|_{\eddi(\varepsilon)}\in C^{d_i}\left({\eddi(\varepsilon)^*}\right)
		\text{\, for every\, }\\ \varepsilon\in E_i\cap H
		\text{ and \;} i \in\{ 1,2, 3\}
		\bigr\}\ .
	\end{multline}
	For $\bd$ as above, we define the linear space of \emph{hierarchical box splines with edge and vertex smoothness $\bd$ on the hierarchical multicell domain $H$} as the set 
	\begin{multline}\label{eq:HEV}
		\hat{\SSS}^\bd(H, \calV_\bn) = \biggl\{f\in\PP(H,\calV_\bn)\,\colon\,
		f|_{U}\in\DD_{I}(U,\calV_\bn) 
		\text{ for }  \triangle,\triangle'\in M^\ell,\\ 
		\text{ for some } \ell, \, \triangle\cap\triangle'\neq \emptyset\ ,\, 
		U = \overline\triangle\cup\overline\triangle' \, , 
		{\mbox{ and $I=\bigcap_{i\in \ST(\triangle,\triangle')} I_{i}^\bd$}}
		\biggr\}.
	\end{multline}
\end{definition}
In Equation \eqref{eq:HH}, the diamond of and edge $\varepsilon \in H$ is taken over the multicell domain $M^\ell$ such that $\varepsilon$ is an edge of the a triangle $\triangle\in M^\ell$. 
Namely,  
\[\displaystyle\eddi(\varepsilon)=\bigcup_{\triangle\in M^\ell, \varepsilon\in\hat\triangle}\hat\triangle\ .\]
Let us fix $\bn=(n_1,n_2,n_3)\in\ZZ_{\ge 1}^3$ and the corresponding box spline $\calB_{\bn}$. 

For each level $\ell=1,\dots,N$ we denote by $B_\bn^\ell$ the set of translates of $\calB_{\bn}$ with respect to the grid $G^\ell$.  
In particular, we have $B_\bn^1=B_\bn(G)$ as defined in Equation (\ref{eq:translates}), and recursively
\[
B_\bn^{\ell+1}=\bigl\{\beta(2\cdot\;)\colon \beta\in B_\bn^\ell\bigr\}\ .
\]
In an analogous way as we introduce $B_\bn(M)$ in Definition \ref{def:active} for a multicell domain $M$ in the three-directional grid $G$, we now define the set $B_\bn^\ell({M^\ell})$ of active box splines on the multicell domain $M^\ell$ in $G^\ell$ as follows,  %This the box splines whose support has a non-empty intersection with $\calM^\ell$ i.e., 
\[
B_\bn^\ell({M^\ell})=\bigl\{\beta|_{M^\ell}\colon \beta\in B_\bn^\ell\ , \text{ and } \supp \beta\cap \calM^\ell\neq \emptyset\bigr\}\ .
\]
A set of linearly independent box splines on a hierarchical multicell domain $H$ can be constructed by a selection procedure analogous to that proposed by Kraft in \cite{kraft1998} in the context of tensor-product B-splines. 

For all levels $\ell$, we select box splines translates, and define the sets $K^\ell$ as follows
\[
K^\ell = \bigl\{\beta^\ell\in B_\bn^\ell(M^\ell)\colon \supp \beta^\ell\cap \calM^{\ell -1}=\emptyset\bigr\}.
\]
The collection of these box splines translates in the levels $\ell=1,\dots, N$ forms a \emph{hierarchical box splines basis} given by
\begin{equation}\label{eq:basisK}
	K=\bigcup_{\ell=1}^N K^\ell.
\end{equation}
The linear independence of the functions in $K$ is implied by the local linear independence of the box splines at each level, see \cite{kraft1998}.

Then the question of completeness of hierarchical type-I box spline spaces can be stated as follows. For a given hierarchical multicell domain $H$, does the basis $K$ in Equation \eqref{eq:basisK} span the hierarchical box spline space $\SSS^\bd(H,\calV_\bn)$ defined in Equation \eqref{eq:HH}? Does $K$ span $\hat{\SSS}^\bd(H, \calV_\bn)$ defined in Equation \eqref{eq:HEV}?

In the following theorem we address the first question, and provide a sufficient condition for the completeness of the hierarchical spline basis.
\begin{theorem}\label{theorem:hierarchicalbasis}
	Let $H$ be a three-directional hierarchical multicell domain, and let $\calB_n$ be a box spline, for some triple $\bn=(n_1,n_2,n_3)\in \Z^3_{\geq 1}$.
	The basis $K$ in Equation (\ref{eq:basisK}) spans the hierarchical box spline space $\SSS^\bd(H,\calV_\bn)$ if each multicell domain $M^\ell$ of $H$ is admissible (Definition~\ref{def:admissible}) with respect to the grid level $\ell$. 
\end{theorem}
\begin{proof}
	The proof follows standard arguments already presented in
	\cites{GJ2013,mokris} for the case of hierarchical tensor B-spline
	bases.  
	Let $\Omega\subset \RR^2$ be the domain associated to $H$.  
	We prove by induction on the levels $\ell$ that every spline function $s\in\SSS^\bd(H,\calV_\bn)$ admits a representation
	\begin{equation}\label{eq:rep}
		s = (h^1 + \cdots + h^{N})|_{\Omega},
	\end{equation}
	where $h^\ell\in\spanset B^\ell_\bn (M^\ell)$, and 
	\begin{equation}\label{eq:168}
		h^\ell|_{\calM^\ell} = s|_{\calM^\ell} -
		(h^1+\ldots+h^{\ell-1})|_{\calM^\ell}.
	\end{equation}
	for $\ell=1,\dots,N$. %This is proved by induction with respect to $\ell$. 
	
	For any given level $\ell$, all functions $h^k|_{\calM^\ell}$ of lower
	levels $k<\ell$ are contained in $\SSS^\bd(M^\ell,\calV_\bn)$. This follows from the relation
	\[
	\spanset B^k_\bn(M^k)|_{M^\ell}\subseteq \SSS^\bd(M^\ell,\calV_\bn).
	\]
	It follows that $s|_{\calM^\ell}\in
	\SSS^\bd(M^\ell,\calV_\bn)$. Consequently, the right-hand side of Equation
	\eqref{eq:168} is contained in $\SSS^\bd(M^\ell,\calV_\bn)$. 
	Since the multicell domain $M^\ell$ is admissible, we conclude that
	$h^\ell\in\spanset B^\ell_\bn(M^\ell)$ according to Corollary
	\ref{cor:compl}. In particular, choosing $\ell=N$ in
	Equation \eqref{eq:168} we get Equation \eqref{eq:rep}.
	
	Moreover, the construction of the functions $h^\ell$ ensures that
	\begin{equation*}
		h^\ell|_{{\calM}^{\ell-1}}=0|_{{\calM}^{\ell-1}}.
	\end{equation*}
	Since the box splines possess the property of local linear
	independence we can conclude that
	$h^\ell\in\spanset K^\ell$. This completes the proof.
\end{proof}
Theorem \ref{theorem:hierarchicalbasis} is a consequence of Corollary \ref{cor:compl}, and therefore of Lemma \ref{lemma:ccl}. This result generalizes the completeness property of the space of translates of the quartic box spline $\calB_{\bm 2}$ proved in \cite{vmj2015}*{Theorem 26}. 

\section{Concluding remarks}\label{sec:conlusion}

\begin{remark}\label{remark1}
	The contact characterization property proved in Lemma \ref{lemma:ccl} implies the completeness of the space spanned by the translates of type-I box splines $\calB_\bn$ on a multicell domain $M$ with respect to the spline space $\Splsp^{\bd(\bn)}(M,\calV_{\bm n})$, where $\bd (\bn)=(n_2 +n_3 - 2, n_1 +n_3 - 2, n_1 +  n_2 - 2)$, and $\bn=(n_1,n_2,n_3)$. 
	This result holds whenever $M$ is an admissible domain i.e., whenever $M$ does not have over-concave vertices nor kissing triangles, and $\supp(\beta)$ is connected for all translates $\beta$ of $\calB_\bn$ which have support on $M$. 
	This is a sufficient condition, but it will also be interesting to prove necessary conditions to achieve this completeness property of the type-I box splines. 
	In this direction, a complete characterization of the vertex-vertex contact plays a crucial role. This calls 
	for exploring the algebraic formulation of super-smoothness at vertices
	in order to proving an analogous to Lemma \ref{lemma:ccl} for
	vertex-vertex contact of type-I box splines.
	
\end{remark}	

\begin{remark} We partially undertake the vertex-vertex contact question raised in the previous remark, for low degrees.
	In particular, we have verified algorithmically the vertex-vertex
	contact (that is, Lemma~\ref{lemma:ccl} for the case
	$\overline\triangle\cap\overline\triangle'=\nu\in V$ and the space
	$\hat{\SSS}^\bd(M, \calV_\bn)$) for all type-I box splines up to
	$\bn=(4,4,4)$.
	
	The approach is based on the fact that both statements $(i)$ and
	$(ii)$ of Lemma~\ref{lemma:ccl} yield linear relations on the coefficients
	$\lambda_\triangle^\beta(f|_\triangle)$ and
	$\lambda_{\triangle'}^\beta(f'|_{\triangle'})$.
	We can express each of the two statements as matrices $A^{(i)}$ and
	$A^{(ii)}$. The latter matrix has the form
	\begin{align*} 
		A^{(ii)} = 
		\left[
		\begin{array}{cccc}
			I_p    & \bm 0_q & -PI_p      & \bm 0_q
		\end{array}
		\right] ,
	\end{align*}
	where $P$ is a permutation matrix, $p = \# B_\bn(\hat\triangle)\cap B_\bn(\hat\triangle')$, and $q= \phi(\bn) - p$.
	
	The matrix\footnote{Observe that $\ker A^{(i)} = \hat{\SSS}^\bd(M, \calV_\bn)$.} $A^{(i)}$ can be computed as a product $A^{(i)} = C_{\nu} L_M$.
	The factor $C_\nu$ describes the B\'ezier continuity conditions at the common vertex $\nu$,
	corresponding to the smoothness type $\ST(\triangle,\triangle')$
	(Definition \ref{def:smooth_type}) and to the regularity vector $\bd =
	(n_2+n_3-2,n_1+n_3-2,n_1+n_2-2)$.
	The right factor $L_M$ is the Bernstein--B\'ezier representation of the translates of
	$\mathcal B_\bn$, considered independently on each of the two triangles of $M$:
	\begin{align*}
		L_M = \left[
		\begin{array}{cc}
			L_{\triangle} & 0 \\
			0 & L_{\triangle'}
		\end{array}
		\right]
	\end{align*}
	or, equivalently, a change of basis matrix of the space $C^{-1}(M^*)$.
	
	Showing that Lemma~\ref{lemma:ccl} holds in this setting (and for a fixed
	$\bn$) is done by showing that $A^{(i)}$ and $A^{(ii)}$ are equivalent
	matrices. Indeed, in all our computations we obtain $A^{(ii)}$ as the
	reduced row echelon form of $A^{(i)}$ for type-I box splines up to total degree $12$.
	This verifies that the matrices are equivalent and so the vertex-vertex contact lemma holds in these cases.
	\if 0
	{\small
		\begin{center}
			\begin{tabular}{|c|c|c|c|c|} \hline
				$\bn$ 
				& \footnotesize{$\{1,2\}$}& \footnotesize{$\{1,3\}$}& \footnotesize{$\{2,3\}$}& \footnotesize{$\{1,2,3\}$}
				\\ \hline
				111 &$1+4$&$1+4$&$1+4$&$1+4$ \\ \hline
				211 &$3+4$&$2+6$&$2+6$&$2+6$   \\ \hline
				221 &$5+6$&$5+6$&$4+8$&$4+8$  \\ \hline
				222 &$8+8$&$8+8$&$8+8$&$7+10$\\ \hline
				311 &$5+4$&$3+8$&$3+8$&$3+8$\\ \hline
				321 &$8+6$&$7+8$&$6+10$&$6+10$\\ \hline
				322 &$12+8$&$11+10$&$11+10$&$10+12$\\ \hline
				331 &$11+8$&$11+8$&$9+12$&$9+12$\\ \hline
				332 &$16+10$&$16+10$&$15+12$&$14+14$\\ \hline
				333 &$21+12$&$21+12$&$21+12$&$19+16$\\ \hline
			\end{tabular} \ \begin{tabular}{|c|c|c|c|c|} \hline
				$\bn$ 
				& \footnotesize{$\{1,2\}$}& \footnotesize{$\{1,3\}$}& \footnotesize{$\{2,3\}$}& \footnotesize{$\{1,2,3\}$}
				\\ \hline
				411 &$7+4$&$4+10$&$4+10$&$4+10$\\ \hline
				421 &$11+6$&$9+10$&$8+12$&$8+12$\\ \hline
				422 &$16+8$&$14+12$&$14+12$&$13+14$\\ \hline
				431 &$15+8$&$14+10$&$12+14$&$12+14$\\ \hline
				432 &$21+10$&$20+12$&$19+14$&$18+16$\\ \hline
				433 &$27+12$&$26+14$&$26+14$&$24+18$\\ \hline
				441 &$19+10$&$19+10$&$16+16$&$16+16$\\ \hline
				442 &$26+12$&$26+12$&$24+16$&$23+18$\\ \hline
				443 &$33+14$&$33+14$&$32+16$&$30+20$\\ \hline
				444 &$40+16$&$40+16$&$40+16$&$37+22$\\ \hline
			\end{tabular}
		\end{center}
	}
	\fi
\end{remark}

\begin{remark}
	In the discussion of completeness of hierarchical type-I box splines spaces there are two main differences to the original approach, which was 
	formulated for tensor-product splines. 
	
	First, the translates of a box spline do not span the whole space of
	bivariate polynomials of a given total degree. For this reason, this
	special polynomial subspace had to be identified. In some sense this
	situation generalizes the tensor-product case, where the B-splines
	span a polynomial space of a given (coordinate-wise) bi-degree,
	instead of the the space of bivariate polynomials of a given total
	degree.
	
	Second, the constraints on the domains are entirely different, due to
	the differences in the characterization of contacts between polynomial
	pieces. For bivariate tensor-product splines, both edge-edge and
	vertex-vertex contacts could be characterized easily by the equality
	of spline coefficients. In the case of type-I box splines, we proved the characterization solely for edge-edge contacts. Consequently, the completeness
	of hierarchical splines requires more severe restrictions to the
	hierarchical grid.
	
	We have alleviated these extra restrictions by proving algorithmically
	the vertex-vertex contact for small polynomial degrees (total degree
	up to 12).
	By proving a characterization of the vertex-vertex contact of box
	splines for arbitrary degree, as we indicated in Remark \ref{remark1},
	we could relax these restrictions on the hierarchical grids and
	construct more general hierarchical type-I spline space. We leave this
	as a future research direction.
\end{remark}

\end{document}